\documentclass[11pt]{article}
\usepackage{color}
\usepackage[colorlinks=true,citecolor=blue]{hyperref}
\usepackage{cite}
\usepackage{amssymb,amsmath,amsfonts,amsthm,amsthm}
\usepackage{mathtools}
\usepackage[english]{babel}
\usepackage[utf8]{inputenc}
\usepackage[margin=0.90in]{geometry}
\usepackage[margin=0.2in]{caption}
\usepackage{graphicx}
\usepackage{times,fourier}
\usepackage[caption = false]{subfig}
\usepackage{extarrows}
\usepackage{tikz}

\makeatletter
\renewcommand\section{\@startsection {section}{1}{\z@}%
	{-2ex \@plus -1ex \@minus -.2ex}%
	{1ex \@plus.1ex}%
	{\normalfont\bf\sffamily}}
\renewcommand\subsection{\@startsection{subsection}{2}{\z@}%
	{-1.75ex\@plus -0.4ex \@minus -.2ex}%
	{0.6ex \@plus .1ex}%
	{\normalfont\small\bf\sffamily}}
\renewcommand\subsubsection{\@startsection{subsubsection}{3}{\z@}%
	{-0.6ex\@plus -0.2ex \@minus -.2ex}%
	{0.4ex \@plus .1ex}%
	{\normalfont\normalsize\it}}
\renewcommand\paragraph{\@startsection{paragraph}{4}{\z@}%
	{0.2ex \@plus0.2ex \@minus0.1ex}{-0.5em}%
	{\normalfont\normalsize\bfseries}}
\def\ps@headings{%
	\let\@oddfoot\@empty
	\let\@evenfoot\@empty
	\def\@evenhead{\small\sffamily\thepage\hfil\slshape\leftmark}%
	\def\@oddhead{\small\sffamily{\slshape\rightmark}\hfil\thepage}%
	\let\@mkboth\markboth
	\def\chaptermark##1{\markboth{{\ifnum \c@secnumdepth >\m@ne
				\if@mainmatter \@chapapp\ \thechapter. \ \fi \fi ##1}}{}}%
	\def\sectionmark##1{\markright {{\ifnum \c@secnumdepth >\z@
				\thesection. \ \fi ##1}}}}
\def\fbf#1{\setbox0=\hbox{$#1$}\kern-0.10\wd0
	\lower0.02em\copy0\kern-\wd0 \lower0.02em\hbox{\kern+0.04em\copy0}\kern-\wd0
	\raise0.00em\copy0\kern-\wd0 \raise0.00em\hbox{\kern-0.04em\box0}}
\makeatother

\advance\textwidth -8mm
\advance\textheight -10mm
\advance\hoffset 4mm
\advance\voffset 4mm
\advance\textheight 0mm
\advance\parskip 4pt
\numberwithin{equation}{section}
1

\newtheorem{theorem}{Theorem}[section]

\newtheorem{remark}[theorem]{Remark}
\newtheorem{proposition}[theorem]{Proposition}

\def\maketitle{\par\noindent{\LARGE\bf\sffamily\thetitle}\\[1.4ex]
	{\large\theauthor}\\[0.6ex]
	\textit{\thetextinfo}\\[0.2ex]
	{\small\today}\par\vglue1.4\bigskipamount}
\def\title#1{\def\thetitle{#1}}
\def\author#1{\def\theauthor{#1}}
\def\textinfo#1{\def\thetextinfo{#1}}

\def\be{\begin{equation}}
	\def\ee{\end{equation}}
\def\bse{\begin{subequations}}
	\def\ese{\end{subequations}}
\newtheorem{RHP}{RH Problem}[section]

\definecolor{deeppurple}{rgb}{0.5, 0, 0.7}

\def\Wr{\mathop{\rm Wr}\nolimits}
\def\gl{\mathrel{\mathpalette\overl@ss>}}

\def\sech{\mathop{\rm sech}\nolimits}

\def\Real{\mathbb{R}}

\def\i{\text{i}}

\def\Re{\mathop{\rm Re}\nolimits}
\def\Im{\mathop{\rm Im}\nolimits}
\def\Res{\mathop{\rm Res}\limits}

\def\d{\mathrm{d}}

\def\e{\mathop{\rm e}\nolimits}
\def\@#1{{\mathbf{#1}}}
\def\_#1{{\mathsf{#1}}}
\let\~=\tilde
\let\==\bar
\let\^=\hat
\let\<=\langle
\let\>=\rangle

\def\min{\mathop{\rm min}\nolimits}
\def\max{\mathop{\rm max}\nolimits}

\def\note[#1]{\marginpar{\color{red}[#1]}}

\def\XXint#1#2#3{{\setbox0=\hbox{$#1{#2#3}{\int}$}
		\vcenter{\hbox{$#2#3$}}\kern-.5\wd0}}

\def\1{{\bf 1}}

\def\e{\mathrm{e}}
\makeatother

\let\trueparagraph=\paragraph
\def\paragraph#1{\par\smallskip\trueparagraph{\rm\textbf{#1}}}
\advance\textheight 12mm
\advance\voffset -6mm
\allowdisplaybreaks

\def\br{\begingroup\color{red}}
\def\er{\endgroup}

\begin{document}
\pagestyle{plain}
\title{\bf Soliton resolution and   asymptotic stability of $N$-soliton solutions for   the  defocusing  mKdV  equation with a non-vanishing background }
\author{\large
	Zechuan Zhang\,$^{1}$, Taiyang Xu\,$^1$, and Engui Fan\,$^1$}
\textinfo
{\normalsize\it
	1: School of Mathematical Sciences, Fudan University, Shanghai 200433, P.R. China\\
	\normalsize\it
	Authors' Email: \{17110180013, tyxu19, faneg\}@fudan.edu.cn}

\maketitle

\kern-4ex
\begin{abstract}
\noindent
	We analytically study the large-time asymptotics of the solution of the 
	defocusing modified Korteweg-de Vries (mKdV) equation under a symmetric non-vanishing background, which supports the emergence of solitons. It is demonstrated that the asymptotic expansion of the solution at the large time could verify the renowned soliton resolution conjecture. Moreover, the asymptotic stability of $N$-soliton solution is also exhibited in the present work.  We establish our results by performing a $\bar{\partial}$-nonlinear steepest descent analysis to the associated Riemann-Hilbert (RH) problem. 
\end{abstract}

\baselineskip=16pt

\tableofcontents%

\section{Introduction}
We investigate the Cauchy problem for the defocusing modified Korteweg-de Vries (mKdV) equation with finite density initial data
\begin{align}
	&q_t(x,t)+q_{xxx}(x,t)-6q^2(x,t)q_x(x,t)=0,\quad\,(x,t)\in\mathbb{R}\times\mathbb{R}^+,\label{mkdvequation}\\
	&q(x,0)=q_0(x) \to\pm1,\quad\, x\rightarrow \pm\infty.\label{boundary}
\end{align}
Let  $\lim_{x\to\pm\infty}q_0(x) = q_\pm$. 
Note that, if $q_\pm = \pm A$ as $x\to\pm\infty$, one can always reduce oneself to either $A = 1$ or $A = -1$ without loss of generality thanks to scaling invariance of mKdV equation. Indeed, one can make $u=A^{-1}q$, $\tilde{x}=Ax$ and $\tilde{t}=A^3t$, it is consequent that $u(\tilde{x},\tilde{t})$ satisfies \eqref{mkdvequation} with the normalized
boundary condition $u_0(\tilde{x})\to\pm 1 $ as $\tilde{x}\to\pm\infty$.
Note also that the case $q_\pm = \mp 1$ is trivially reduced to the present one thanks to the invariance of the mKdV equation under change of sign [i.e., the transformation $q(x,t) \mapsto - u(x,t)$ ].

The mKdV equation arises in various physical fields, such as acoustic wave and phonons in a certain anharmonic lattice \cite{RN1,RN2}, Alf\'{e}n wave in a cold collision-free plasma \cite{RN3}, meandering ocean currents \cite{RN4}, hyperbolic surfaces \cite{RN5}, and Schottky barrier transmission \cite{RN6}.

There are plenty of results on the mathematical properties for the mKdV equation.
Here we cite only those that are closed to our study.  In the late 1970s,  the inverse scattering  theory was  applied to  solve  the mKdV equation and investigate large-time asymptotics for the mKdV equation.
For example, Wadati investigated the focusing mKdV equation with zero boundary conditions   and derived simple-pole, double-pole and triple-pole solutions \cite{RN7,RN8}.
The long-time behavior of the defocusing mKdV equation with given Schwartz  initial data is provided by  Segur and Ablowitz without consideration of  solitons \cite{RN21}.
Deift and Zhou developed nonlinear steepest descent method  and obtained the long-time asymptotic behavior of the defocusing mKdV equation with  the  Schwartz initial data  in their seminal work \cite{RN9}.
This approach was  further developed   into  a $\bar\partial$ steepest descent method    by McLaughlin and   Miller
to analyze asymptotics of orthogonal polynomials with non-analytical weights  \cite{MandM2006,MandM2008}.  Later, with Dieng, they applied it to investigate the defocusing NLS equation under essentially minimal regularity assumptions on finite mass initial data \cite{DandMNLS}.
Boutet de Monvel \emph{et al.}  discussed the initial boundary value problem of  defocusing mKdV equation on the half line
by using the Fokas method \cite{Monvel}.   For the weighted Sobolev  initial data,     Chen   and Liu \emph{et al.}  have  studied the large-time asymptotic behavior of defocusing mKdV equation with zero boundary conditions  without consideration of solitons  \cite{chenliu1}, and the long-time asymptotic behavior of  focusing mKdV equation with zero boundary conditions  with   solitons \cite{chenliu2}.
However, for defocusing mKdV equation with nonzero boundary conditions,  soliton solutions  will appear due to  non-empty discrete spectrum
for finite mass initial data.    It  is necessary to consider the effect of soliton solutions when we study large-time asymptotic behavior, which naturally require a more detailed  necessary  description to obtain the large-time asymptotics of the  defocusing  mKdV equation.

Our goal in this paper is to give detailed asymptotic analysis for the  defocusing  mKdV equation  (\ref{mkdvequation})  with finite density type initial data in the given space-time solitonic regions $|x/t+4|<2$; see Figure \ref{cone} for an illustration.
\begin{figure}[htbp]
    \begin{center}
    \begin{tikzpicture}[node distance=2cm]
    \draw[-latex](-5.5,0)--(5,0)node[right]{$x$};
    \draw[-latex](0,0)--(0,3)node[above]{$t$};
    \draw[ red](0,0)--(-5,5/6)node[left,black]{\scriptsize{$\xi=-6 $}};
    \draw[ red](0,0)--( -5,2.5)node[above,black]{\scriptsize{$\xi=-2 $}};

    \node[below]{$0$};
    \coordinate (A) at (-5.2, 0.2);
	\fill (A) node[right] {\scriptsize{Solitonless region}};
    \coordinate (C) at (-5.5, 1.4);
    \fill (C) node[right] {\scriptsize{Solitonic region}};
     \coordinate (D) at (3, 1.2);
	 \fill (D) node[left] {\scriptsize{Solitonless region}};
    \coordinate (E) at (-0.4, 1.8);
    \end{tikzpicture}
    \caption{\small  The $(x,t)$-plane is divided into  three kinds of asymptotic regions:
    Solitonic region, $-6 <\xi\le-2 $; Solitonless region,  $\xi<-6 $ and  $\xi>-2 $;  Transition region, $\xi\approx -6 $.Here, $\xi:=x/t$.} \label{cone}
    \end{center}
\end{figure}

We  investigate  the  asymptotic stability and soliton resolution  for the mKdV equation  (\ref{mkdvequation}) for the region  $|x/t+4| <2$,
in which there are no  phase points on the real axis.
For the case of solitonless regions $|x/t+4|> 2$, it is considered in \cite{xzf}.
Based on the phase velocity $\xi$, apart from the critical line at $\xi = -6$, the other critical line should technically be at $\xi = 6$ (refere the Propostion \ref{relation of xi}). However, for the region where $x/t$ belongs to $[-2,6)$, we find that the set of solitons contributing to the asymptotic behavior is empty. This indicates that the region $[-2,6)$ can be seen as a special case of the solitonic region by setting the index $\Lambda$ defined in \eqref{def: Lambda} be empty,  and can also be classified as a solitonless region. Here, we consider the region $[-2, 6)$ as part of the solitonless region, which leads \br to \er the classification shown in Figure \ref{cone}. The only transient region in Figure \ref{cone} is near $\xi = -6$, which we have made a discussion in \cite{WXF2023}.

The soliton resolution conjecture is one of the most interesting phenomenon
observed in the study of solutions to certain nonlinear dispersive partial differential equations (PDEs). The conjecture suggests that solutions with generic initial data for many dispersive equations should eventually decompose into a finite number of solitons, each moving at different speeds, along with a radiative term \cite{RN14,RN15,RN16,RN17,RN18,RN19}. Understanding and proving soliton resolution contribute to our broader understanding of the behavior of nonlinear dispersive systems, shedding light on the intricate interplay between nonlinearities, dispersion, and soliton dynamics.
Recently, large-time asymptotics and  soliton resolution   for some integrable systems have  been
obtained  by using $\bar{\partial}$-generalization of the nonlinear steepest descent method
\cite{fNLS,Liu3,RN20,LJQ,YF1,YYLmch}.

This paper is organized as follows. In Section \ref{sec:pre}, we get down to the spectral analysis on the Lax pair.
We  state  the symmetries, asymptotic behaviors and time evolution of the scattering data. 
Further discussion show that the zeros of $a(z)$ are  simple and finite.
In Section \ref{sec:deform}, we set up an RH problem for a sectionally meromorphic function $m(z)$  comprised by the Jost solutions and the scattering data. 
Once the solution of the RH problem exists, we can directly obtain the reconstruction formula.
To handle the RH problem, we first give the distributions of phase points and the signature table of $\Re(2\i t\theta)$, then introduce a set of conjugations and interpolations, such that the original 
RH problem  becomes a standard RH problem.
Subsequently, according to the basic factorization of the jump matrix,
we introduce some  appropriate extensions to deform the jumps onto four different contours in the complex plane on which their forms are asymptotically small.
In Section \ref{sec:LTA},   
by  neglecting  the $\bar{\partial}$ term of $m^{(2)}(z)$, we get a conjugation of the RH problem related to the $N$-soliton with the modified scattering data.
In this way, we can consider the asymptotic behavior of  $N$-soliton solutions by using a small norm theorem.
The existence of $m^{(3)}(z)$ is verified, as well as its asymptotic estimate.
Finally, in Section \ref{sec:proofs}, we give the proofs of the main theorems applying the above consequences.

\paragraph{Notation.} We first provide some notations used in this paper:
\vspace{-0.1in}
\begin{itemize}
	\item $\mathbb{R}^+=(0,\infty)$, $\mathbb{C}^{\pm}=\{z\in\mathbb{C}:\pm\Im z>0\}$.
	\item The Japanese bracket is defined as $\langle x\rangle:=\sqrt{1+|x|^2}$.
	\item The normed space  $L^{p,s}(\mathbb{R})$ is defined with $\|q\|_{L^{p,s}(\mathbb{R})}\doteq\|\langle x\rangle^sq\|_{L^{p}(\mathbb{R})}$;
	$W^{k,p}(\mathbb{R})$ is defined with $\|q\|_{W^{k,p}(\mathbb{R})}$$\doteq\sum_{j=0}^k\|\partial^jq\|_{L^{p}(\mathbb{R})}$;
	$H^k(\mathbb{R})$  is defined with $\|q\|_{H^k(\mathbb{R})}\doteq\|\langle x\rangle^k\hat{q}\|_{L^2(\mathbb{R})}$, where $\hat{u}$ is the Fourier transform of $u$, and $H^{k,k}(\mathbb{R})\doteq L^{2,k}(\mathbb{R})\cap H^{k}(\mathbb{R})$.
	\item $\sigma_i (i=1, 2, 3)$ are the Pauli matrices defined as 
	\be
	\sigma_1=\begin{pmatrix}0&1\\
		1&0\end{pmatrix},\quad\, \sigma_2=\begin{pmatrix}0&-\i\\\i&0\end{pmatrix}\quad\, \sigma_3=\begin{pmatrix}1&0\\0&-1\end{pmatrix}.
	\ee
\end{itemize}

\paragraph{Main results.}The main results of the work are listed below.
\begin{theorem}\label{mainresult1}
	Suppose the initial data $q_0\mp1\in H^{4,4}(\mathbb{R}^{{\pm}})$ with scattering data $\left\{r(z),\{z_j,c_j\}_{j=0}^{N-1}\right\}$. 
	Order $z_j$ such that
	\begin{equation}\label{order}
		\Re z_0>\Re z_1>...>\Re z_{N-1}\geq0,    
	\end{equation}
	and define $\xi=\frac{x}{t}$. Let $q^{(sol),N}(x,t)$ be the $N$-soliton solution whose scattering data can be denoted by $\{\tilde{r}\equiv0,\{z_j,\tilde{c}_j\}_{j=0}^{N-1}\}$, where
	\begin{equation}
		\tilde{c}_j=c_j\exp\left(-\frac{1}{\i\pi}\int_{\mathbb{R}}\log(1-|r(s)|^2)\left(\frac{1}{s-z_j}-\frac{1}{2s}\right)ds\right).
	\end{equation} For fixed $\xi_0\in(0,2)$, there exist constants $t_0=t_0(q_0,\xi_0)$ and $C=C(q_0,\xi_0)$ such that the potential $q(x,t)$ of $(\ref{lax0})$ satisfies
	\begin{equation}
		|q(x,t)-q^{(sol),N}(x,t)|\leq ct^{-1},  \qquad t>t_0,\quad |\xi+4|\leq\xi_0.
	\end{equation}
	Furthermore, for $t>t_0$ and $|\xi+4|<\xi_0$, we have confirmed the soliton resolution for the $N$-soliton solution
	\begin{equation}\label{qse}
		q(x,t)=-1+\sum_{j=0}^{N-1}[sol(z_{j};x-x_j,t)+1]+\mathcal{O}(t^{-1}),
	\end{equation}
	where $sol(z_{j};x-x_j,t)$ is defined by (\ref{sol}), and
	\begin{equation}
		x_{j}=\frac{1}{2\Im z_{j}}\left\{\log\left(\frac{|c_{j}|}{\Im z_{j}}\prod_{k\in\triangle,k\neq j}\left|\frac{(z_{j}-z_k)(z_j+\bar{z}_k)}{(z_{j}z_{k}-1)(z_{j}\bar{z}_k+1)}\right|\right)-\frac{\Im z_{j}}{\pi}\int_{\mathbb{R}}\frac{\log(1-|r(s)|^2)}{|s-z_{j}|^2}ds\right\}.\nonumber
	\end{equation}
\end{theorem}

As a corollary of Theorem \ref{mainresult1}, we have the following theorem:
\begin{theorem}\label{mainresult2}
	Let $q^{(sol),M}(x,t)$ be an $M$-soliton satisfying the boundary conditions in ($\ref{boundary}$). Let $\left\{0,\{z_j,c_j\}_{j=0}^{M-1}\right\}$ denote its reflectionless scattering data. Then there exist $\varepsilon_0$ and $C>0$, for any initial data $q_0$ of problem \eqref{mkdvequation}-\eqref{boundary} satisfying
	\begin{equation}
		\varepsilon\doteq\|q_0-q^{(sol),M}(x,0)\|_{H^{4,4}(\mathbb{R})}<\varepsilon_0,
	\end{equation}
	$q_0$ generates scattering data $\{r',\{z_j',c'_j\}_{j=0}^{N-1}\}$ with $N\geq M$. For the two discrete spectrum we use the same order as in \eqref{order}.
	Suppose $M$ poles in the discrete spectrum of $q_0$ are close to that of $q^{(sol),M}(x,t)$. The remaining are close to $\pm1$. It is to say, there exists an index $L\in\{0,...,N-1\}$ with $L+M\leq N-1$, so that we have
	\begin{equation}\label{M1}
		\max_{0\leq j\leq M-1}(|z_j-z_{j+L}'|+|c_j-c_{j+L}'|)+\max_{j>M+L}|1+z_j'|+\max_{j<L}|1-z_{j}'|<C\varepsilon.
	\end{equation}
	Define $\xi=x/t$ and let $\xi_0\in(0,2)$ so that $\{\Re  z_j\}_{j=0}^{M-1}\subset[0,\frac{\xi_0}{2})$. Then we have constants $t_0(q_0,\xi_0)>0$, $C=C(q_0,\xi_0)>0$ and $\{x_{k+L}\}_{k=0}^{M-1}\subset\mathbb{R}$ such that for $t>t_0(q_0,\xi_0)$, $|\xi+4|<\xi_0$,
	\begin{equation}\label{qse2}
		\left|q(x,t)-\left[-1+\sum_{j=0}^{M-1}[sol(z_{j+L}';x-x_{j+L},t)+1]\right]\right|\leq Ct^{-1}.
	\end{equation}
\end{theorem}

\section{Preliminaries} \label{sec:pre}
In this section we provide an review of the results on the direct and inverse scattering problem for the mKdV equation \eqref{mkdvequation}.

In section \ref{subsec:Lax pair and Jost solutions}, we review the Lax pair of formulation 
of the mKdV and equation, and we present the properties of the Jost solutions. In section \ref{subsec:RH formulation}, we introduce the basic RH problem, which serves as the basis of the nonlinear steepest approach. In section \ref{subsec:signature table},  we plot the signature tables of the saddle functions so that we can open $\bar{\partial}$ lens in accordance with the corresponding decay regions as $t$ sufficiently large. 

\subsection{Lax pair and Jost solutions} \label{subsec:Lax pair and Jost solutions}

The defocusing mKdV equation \eqref{mkdvequation} is the compatibility condition of the following Lax pair 
\begin{equation}\label{lax0}
	\psi_x=X\psi,\quad\,  \psi_t =T\psi,
\end{equation}
where $\psi=\psi(\lambda;x,t)$ is a matrix-valued  eigenfunction,  $\lambda\in\mathbb{C}$ is the spectral  parameter, and
\be
	X=X(\lambda;x,t)=\i\lambda \sigma_3+Q,\qquad
	T=T(\lambda;x,t)=4\lambda^2 X-2\i\lambda\sigma_3(Q_x-Q^2)+2Q^3-Q_{xx},
\ee
with  $Q=\begin{pmatrix} 0&q(x,t)\\q(x,t)&0\end{pmatrix}$.
\paragraph{Existence and differentiability  of Jost functions.} 
Taking the non-zero boundary conditions \eqref{boundary},
we then get the asymptotic spectral problems
\begin{equation}
	\phi^{\pm}_x=X_\pm \phi^{\pm},\quad\,  \phi^{\pm}_t=T_{\pm} \phi^{\pm}, \label{lax5}
\end{equation}
with
$X_{\pm}(z;x)=\i\lambda\sigma_3+Q_{\pm}$, $T_{\pm}(z;x)=(4\lambda^2+2)X_{\pm}$, and $Q_{\pm}=\pm \sigma_1$.

The asymptotic eigenvector matrix is given by
\begin{align}
	Y_{\pm}(z)=I\mp \frac{1}{z}\sigma_2,
\end{align}
where $I$ denotes the $2\times2$ identity matrix, and $z$ is the uniformization variable defined as $z=\lambda+\zeta$, with  $\lambda(z)=\frac{1}{2}(z+z^{-1})$, and $\zeta(z)=\frac{1}{2}(z-z^{-1})$. For reference, note that $\det Y_{\pm}(z)=1-\frac{1}{z^2}$.

As usual, we define the  Jost eigenfunctions $ \psi^{\pm}(z;x)$ as the solutions of the scattering problem such that
$$  \psi^{\pm}(z;x)=Y_\pm(z)\\e^{\i\zeta(z)x\sigma_3}+o(1), \qquad x \rightarrow  \pm \infty.$$
Subsequently, the modified eigenfunctions $\mu^\pm(z;x)$ can be given by factorizing the asymptotic exponential oscillations:
\be
	 \mu^{\pm}(z;x) = \psi^{\pm}(z;x) \e^{-\i\zeta(z)x\sigma_3}. \label{trasform}
\ee
Furthermore, $\mu^{\pm}(z;x)$ can be defined by the following Volterra integral equations
\begin{equation}
	\mu^{\pm}(z;x)=\begin{cases}
		Y_{\pm}(z)+\int_{\pm\infty}^{x}Y_{\pm}(z)\e^{\i\zeta(z)(x-y)\hat{\sigma}_3}\big[Y_{\pm}^{-1}(z)\Delta Q_{\pm}(y)\mu_{\pm}(z;y)\big]dy,& {z\neq\pm 1,} \\
		Y_{\pm}(z)+\int_{\pm\infty}^{x}\big[I+(x-y)(Q_{\pm}\pm \i\sigma_3)\big]\Delta Q_{\pm}(y)\mu_{\pm}(z;y)dy,& {z=\pm 1,} \nonumber
	\end{cases}
\end{equation}
where $\Delta Q_{\pm}=Q-Q_\pm$.
Hereafter, we use $\mu_{i}^{\pm}(z;x)$ to denote the $i$-th column of $\mu^\pm(z;x)$.
Below  we just take a quick review on  some existing properties for  the Jost functions $\mu^{\pm}(z;x)$,
which can be shown in similar way to the reference  \cite{RN20}. 

\begin{proposition}\label{analydiff}
	Given $n\in\mathbb{N}_0$, let $q(x)\mp1\in$ $L^{1,n+1}(\mathbb{R}^\pm)$, $q'(x)\in$ $W^{1,1}(\mathbb{R})$.
	\begin{itemize}
		\item[$\blacktriangleright$]  For $z\in\mathbb{C}\setminus\{0\}$, $\mu_1^+(z;x)$ and $\mu_2^-(z;x)$ can be analytically extended to $\mathbb{C}^+$ and continuously extended to $\mathbb{C}^+\cup \mathbb{R}$; $\mu_1^-(z;x)$ and $\mu_2^+(z;x)$ can be analytically extended to $\mathbb{C}^-$ and continuously extended to $\mathbb{C}^-\cup \mathbb{R}$.
		\item[$\blacktriangleright$] The maps $q(x)\rightarrow$ $\frac{\partial^n}{\partial z^n}\mu^{\pm}_{i}(z)$ $(i=1,2)$ are Lipschitz continuous, specifically, for any $x_0\in\mathbb{R}$, $\mu_1^-(z)$ and $\mu_2^+(z)$ are continuously differentiable mappings:
		\begin{align}\label{21}
			&\partial_z^n\mu_1^-: \bar{\mathbb{C}}^-\setminus\{0\} \rightarrow L^{\infty}_{loc}\{\bar{\mathbb{C}}^-\setminus\{0\}, C^1((-\infty,x_0],\mathbb{C}^2)\cap W^{1,\infty}((-\infty,x_0],\mathbb{C}^2)\},\\
			&\partial_z^n\mu_2^+: \bar{\mathbb{C}}^-\setminus\{0\} \rightarrow L^{\infty}_{loc}\{\bar{\mathbb{C}}^-\setminus\{0\}, C^1([x_0,\infty),\mathbb{C}^2)\cap W^{1,\infty}([x_0,\infty),\mathbb{C}^2)\},
		\end{align}
		$\mu_1^+(z)$ and $\mu_2^-(z)$ are continuously differentiable mappings:
		\begin{align}\label{22}
			&\partial_z^n\mu_1^+: \bar{\mathbb{C}}^+\setminus\{0\} \rightarrow L^{\infty}_{loc}\{\bar{\mathbb{C}}^+\setminus\{0\}, C^1([x_0,\infty),\mathbb{C}^2)\cap W^{1,\infty}([x_0,\infty),\mathbb{C}^2)\}, \\
			&\partial_z^n\mu_2^+: \bar{\mathbb{C}}^+\setminus\{0\} \rightarrow L^{\infty}_{loc}\{\bar{\mathbb{C}}^+\setminus\{0\}, C^1((-\infty,x_0],\mathbb{C}^2)\cap W^{1,\infty}((-\infty,x_0],\mathbb{C}^2)\}.
		\end{align}
	\item[$\blacktriangleright$] For  the map   $q(x)\rightarrow \partial_z^n\mu_1^+(z)$, there exists an increasing function $ F_n(t)$, such that
	\begin{equation}
		|\partial_z^n\mu_1^+(z)|\leq F_n[(1+|x|)^{n+1}\| q-1\|_{L^{1,n+1}(x,\infty)}],\quad\, z\in \bar{\mathbb{C}}^+\setminus\{0\}.\nonumber
	\end{equation}
	Additionally, given potentials $q(x)$ and $\widetilde{q}(x)$  close enough,  we have
	\begin{equation}
		|\partial_z^n ( \mu_1^+(z)-\widetilde{\mu}_1^+(z) ) | \leq  \|q - \widetilde{q} \|_{L^{1,n+1}(x,\infty)} F_n[(1+|x|)^{n+1}\| q-1\|_{L^{1,n+1}(x,\infty)}].
	\end{equation}
	
	\item[$\blacktriangleright$] Let $S$ be a compact neighborhood of $\{-1,1\}$ in $\bar{\mathbb{C}}^+\setminus\{0\}$. Set $x^{\pm}=\max\{\pm x,0\}$, then there would be a constant $C$ so that
	\begin{equation}
		|\mu_1^+(z)-(1,z^{-1})^T|\leq C\langle x^-\rangle \e^{C\int_x^\infty \langle y-x\rangle|q-1|dy}\|q-1\|_{L^{1,1}}(x,\infty),\quad\,z\in S.
	\end{equation}
\end{itemize}
\end{proposition}

\paragraph{Asymptotics and symmetries  of the Jost functions.} 

The following proposition gives the asymptotic behaviour of the Jost functions.
\begin{proposition}
	Suppose that $q(x)\mp1\in$ $L^{1,n+1}(\mathbb{R}^\pm)$  and   $q'(x)\in$ $W^{1,1}(\mathbb{R})$. Then as $z\rightarrow\infty$,  we have the asymptotics on $\mathbb{C}^+$
	\begin{align}
		&\mu_1^+(z)=e_1+\frac{1}{z}\begin{pmatrix}
			-\i\int_{x}^{\infty}(q^2-1)dx\\
			-iq
		\end{pmatrix}+\mathcal{O}(z^{-2}),\nonumber\quad
		&\mu_2^-(z)=e_2+\frac{1}{z}\begin{pmatrix}
			iq\\
			\i\int_{-\infty}^{x}(q^2-1)dx
		\end{pmatrix}+\mathcal{O}(z^{-2}),\nonumber
	\end{align}
	and on $\mathbb{C}^-$
	\begin{align}
		&\mu_1^-(z)=e_1+\frac{1}{z}\begin{pmatrix}
			-\i\int_{-\infty}^{x}(q^2-1)dx\\
			-\i q
		\end{pmatrix}+\mathcal{O}(z^{-2}),\nonumber\quad
		&\mu_2^+(z)=e_2+\frac{1}{z}\begin{pmatrix}
			\i q\\
			\i\int_{x}^{\infty}(q^2-1)dx
		\end{pmatrix}+\mathcal{O}(z^{-2}).\nonumber
	\end{align}
For $z\in\mathbb{C}^+$,  we have the following asymptotics as $z\rightarrow 0$,
\begin{equation}
	\mu_1^+(z)=-\frac{\i}{z}e_2+\mathcal{O}(1),\quad\,\mu_2^-(z)=-\frac{\i}{z}e_1+\mathcal{O}(1);
\end{equation}
and for $z\in\mathbb{C}^-$,
\begin{equation}
	\mu_1^-(z)=\frac{\i}{z}e_2+\mathcal{O}(1),\quad\,\mu_2^+(z)=\frac{\i}{z}e_1+\mathcal{O}(1).
\end{equation}
\end{proposition}

Abel's theorem indicates that for any solution $\psi(z,x)$ of \eqref{lax0}, one has $\partial_x(\det \psi)=\partial_t (\det\psi)=0$. Thus, both $\psi_-$ and $\psi_+$ are two fundamental matrix solutions of the scattering problem with $\det\psi^{\pm}(z)=1-z^{-2}$, and satisfy the linear relation:
\begin{equation}
	\psi^+(z;x)=\psi^-(z;x)S(z), \quad\, z\in\mathbb{R}\setminus\{\pm1,0\},\label{sydd}
\end{equation}
where $S(z)$ is called scattering matrix and is represented as:
\begin{equation}
S(z)=\begin{pmatrix}
	a(z)&c(z)\\
	b(z)&d(z)
\end{pmatrix}.  \nonumber
\end{equation}
Hereafter, we use a bar to denote the complex conjugate. Then we can establish the following symmetries for the Jost functions and the scattering matrix, which will enable us to set up an RH problem with desirable symmetries. 
\begin{proposition}\label{symmetries}
	Suppose that $q(x)\mp1\in$ $L^{1,n+1}(\mathbb{R}^\pm)$ and   $q'(x)\in$ $W^{1,1}(\mathbb{R})$, then
	\begin{enumerate}
		
		\item For $z\in\bar{\mathbb{C}}^+\setminus\{0\}$, the Jost functions $\psi_j^{\pm}$ $(j=1,2)$ satisfy the symmetries
		\begin{align}
			&\psi_1^{\pm}(z)=\sigma_1\overline{\psi_2^{\pm}(\bar{z})}, \quad\,
			\psi_2^{\pm}(z)=\sigma_1\overline{\psi_1^{\pm}(\bar{z})}.\label{symmetry12}\\
			&\psi_1^{\pm}(z)=\overline{\psi_1^{\pm}(-\bar{z})}, \quad\, \psi_2^{\pm}(z)=\overline{\psi_2^{\pm}(-\bar{z})}.\\
			& \psi_1^{\pm}(z)=\mp\frac{i}{z}\psi_2^{\pm}(\frac{1}{z}), \quad\, \psi_2^{\pm}(z)=\pm\frac{i}{z}\psi_1^{\pm}(\frac{1}{z}). \label{symmetry3}
		\end{align}
		
		\item    The  scattering data  $a(z), b(z), c(z)$ and $  d(z) $ satisfy the symmetries
		\begin{align}
			&S(z)= \sigma_1 \overline{S( \bar z )}\sigma_1=  -\sigma_2S(z^{-1})\sigma_2, \quad\, S(z)=\overline{S(-\bar{z})}.\label{symmetry22}
		\end{align}
	\end{enumerate}
\end{proposition}
It then follows that the scattering matrix $S(z)$  can be rewritten as
\begin{equation}
	S(z)=\left(\begin{array}{cc}
		a(z)&\overline{b(z)}\\
		b(z)&\overline{a(z)}
	\end{array}\right),\quad\, z\in\mathbb{R}\setminus\{\pm1,0\}. \nonumber
\end{equation}

\paragraph{Scattering map from initial data to reflection coefficient}

The reflection coefficient that will be used in the inverse problem is defined by the scattering coefficients $a(z)$ and $b(z)$ 
\begin{equation}
	r(z)\doteq\frac{b(z)}{a(z)}.
\end{equation}
The following proposition provides some useful properties of $a(z)$ and $b(z)$.
\begin{proposition}\label{scattering properties}
	Let $q(x)\mp1\in$ $L^{1,n+1}(\mathbb{R}^\pm)$  and   $q'(x)\in$ $W^{1,1}(\mathbb{R})$,   then
	\begin{enumerate}
		\item The scattering coefficients can be expressed as
		\begin{equation}\label{scatteringdataexpression}
			a(z)=\frac{\Wr(\psi_1^+,\psi_2^-)}{1-z^{-2}},\quad\, b(z)=\frac{\Wr(\psi_1^-,\psi_1^+)}{1-z^{-2}}.
		\end{equation}
		Thus it follows from the analyticities of $\psi^\pm$ that $a(z)$ is analytic in $\mathbb{C}^+$ while $b(z)$ and $r(z)$ are defined on $\mathbb{R}\setminus\{\pm1,0\}$.
		\item For $z\in\mathbb{R}\setminus\{\pm1,0\}$, we have
		\begin{equation}
			|a(z)|^2-|b(z)|^2=1,\label{a>1} 
		\end{equation}
		which gives a constraint
		\begin{equation}\label{r<1}
			|r(z)|^2=1-\frac{1}{|a(z)|^2}<1.
		\end{equation}
		\item  The scattering data has the following asymptotics
		\begin{align}
			&\lim_{z\rightarrow\infty}(a(z)-1)z=\i\int_{\mathbb{R}}(q^2-1)dx, \ z\in\bar{\mathbb{C}}^+, \label{asymptoticsfora1}\\
			&\lim_{z\rightarrow0}(a(z)+1)z^{-1}=\i\int_{\mathbb{R}}(q^2-1)dx, \ z\in\bar{\mathbb{C}}^+\label{asymptoticsfora2}
		\end{align}
		\begin{align}
			&|b(z)|=\mathcal{O}(|z|^{-2}),\quad\, \text{as }|z|\rightarrow\infty,\ \ z\in\mathbb{R}\label{asymptoticsforb1}\\
			&|b(z)|=\mathcal{O}(|z|^{2}),\quad\, \text{as }|z|\rightarrow0, \ z\in\mathbb{R}.\label{asymptoticsforb2}
		\end{align}
		So that
		\begin{align}\label{asymptoticsofr}
			r(z)\sim z^{-2}, \quad\,|z|\rightarrow\infty;\qquad  r(z)\sim z^2, \quad\, |z|\rightarrow0.
		\end{align}
	\end{enumerate}
\end{proposition}

 It can also be shown that  $z=\pm1$ are the simple poles of  $a(z)$ and $b(z)$ on account of the expressions in (\ref{scatteringdataexpression}). Even though, we still can derive the boundness of $r(z)$ at $z=\pm1$. Using the symmetry in (\ref{symmetry3}), it is easy to verify that $\psi_1^-(\pm1)=\pm \i\psi_2^-(\pm1)$, which implies that
\begin{equation}\label{scatteringdata1}
	a(z)=\frac{\pm a_{\pm}}{z\mp1}+\mathcal{O}(1),\quad\, b(z)=-\frac{\i a_\pm}{z\mp1}+\mathcal{O}(1),
\end{equation}
where
\begin{equation}
	a_{\pm}=\frac{1}{2}\det[\psi_1^+(\pm1),\psi_2^-(\pm1)].\label{apm}
\end{equation}
Consequently, it is obvious that the reflection coefficient $r(z)$ is bounded at $z=\pm1$ and
\begin{equation}\label{r1}
	\lim_{z\rightarrow\pm1}r(z)=\mp \i.
\end{equation}
Although the scattering coefficients have some simple poles, given specific conditions on the initial potential, the reflection coefficient will exhibit smoothness and decay.
\begin{proposition}\label{rinH1}
	Suppose that $q(x)\mp1\in L^{1,2}(\mathbb{R}^\pm)$, $q'(x)\in W^{1,1}(\mathbb{R})$,  then $r(z)\in H^{1}(\mathbb{R})$.
\end{proposition}

\begin{proof}
	Proposition  \ref{scattering properties} implies that $a(z)$ and $b(z)$ are continuous when $z\in\mathbb{R}\setminus\{\pm1,0\}$. Then $r(z)$ is continuous for $z\in\mathbb{R}\setminus\{\pm1,0\}$. From \eqref{asymptoticsofr} and \eqref{r1} we know that $r(z)$ is bounded in the small neighborhood of $\{\pm1,0\}$ and $r(z)\in L^1(\mathbb{R})\cap L^2(\mathbb{R})$.
	Here we just need to prove that $r'(z)\in L^2(\mathbb{R})$. For $\delta_0>0$ small, from Proposition \ref{analydiff}, the maps
	\begin{equation}
		q\to\det[\psi_1^+(z),\psi_2^-(z)]\quad\,\text{and}\quad\,q\rightarrow\det[\psi_1^-(z),\psi_1^+(z)]
	\end{equation}
	are locally Lipschitz maps from
	\begin{equation}\label{domain1}
		\{q:q'(x)\in W^{1,1}(\mathbb{R}) \text{ and } q\in L^{1,n+1}(\mathbb{R})\}\rightarrow W^{n,\infty}(\mathbb{R}\setminus(-\delta_0,\delta_0)) \text{ for }n\geq0.
	\end{equation}
	According to Proposition \ref{analydiff}, $q\rightarrow\psi_1^+(z,0)$ is a locally Lipschitz map, and this fact also holds for $q\rightarrow\psi_2^-(z,0)$ and $q\rightarrow\psi_1^-(z,0)$. Together with \eqref{asymptoticsfora1}-\eqref{asymptoticsforb2}, we derive that $q\rightarrow r(z)$ is a locally Lipschitz map with values in $W^{n,\infty}(I_{\delta_0})\cap H^n(I_{\delta_0})$,
	where $I_{\delta_0}\doteq\mathbb{R}\setminus((-\delta_0,\delta_0)\cup(1-\delta_0,1+\delta_0)\cup(-1-\delta_0,-1+\delta_0)).$
	Then fix $\delta_0>0$ small to make sure  that the three intervals dist$(z,\{\pm1\})\leq\delta_0$ and $|z|\leq\delta_0$ have no intersection. In the complement of the union
	\begin{equation}
		|\partial_z^jr(z)|\leq C_{\delta_0}\langle z\rangle^{-1}\text{ for }j=0,1.
	\end{equation}
	In the following step, we will just prove the boundedness of $r'(z)$ in the small neighborhood of $z=1$.
	Let $z\in U(1,\delta_0)$ be a neighborhood of $1$, then we have
	\begin{equation}
		r(z)=\frac{b(z)}{a(z)}=\frac{\det[\psi_1^-,\psi_1^+]}{\det[\psi_1^+,\psi_2^-]}=\frac{\int_1^z\partial_s\det[\psi_1^-(s),\psi_1^+(s)]ds-2ia_+}{\int_1^z\partial_s\det[\psi_1^+(s),\psi_1^-(s)]ds+2a_+},
	\end{equation}
	where $a_\pm$ is defined in (\ref{apm}).
	If $a_+\neq0$ then obviously $r'(z)$ exist and is bounded near $1$.
	If $a_+=0$, then $z=1$ is not a pole of the scattering data $a(z)$ and $b(z)$. Therefore, they are continuous at $z=1$, implying 
	\begin{equation}
r(z)=\frac{\int_1^z\partial_s\det[\psi_1^-(s),\psi_1^+(s)]ds}{\int_1^z\partial_s\det[\psi_1^+(s),\psi_1^-(s)]ds}.
	\end{equation}
	From (\ref{scatteringdataexpression}) we have that
	\begin{equation}\label{a2}
		(z^2-1)a(z)=z^2\det[\psi_1^+,\psi_2^-].
	\end{equation}
	While $a_+=0$ implies that $\det[\psi_1^+,\psi_2^-]|_{z=1}=0$, differentiating (\ref{a2}) at $z=1$ we get
	\begin{equation}
		2a(1)=\partial_z\det[\psi_1^+,\psi_2^-]|_{z=1}.
	\end{equation}
	With $|a(1)|^2=1+|b(1)|^2\geq1$, we have $\partial_z\det[\psi_1^+,\psi_2^-]|_{z=1}\neq0$. It then follows that the derivative $r'(z)$ is bounded in a neighborhood of $1$.
	
	The proof for the case $z=-1$ is just an analogue. For $z=0$ we need to use $r(z^{-1})=-\bar{r}(z)$ to illustrate that $r(z)\rightarrow0$. Consequently, $r'(z)\in L^2(\mathbb{R})$.
\end{proof}


The following proposition gives a bound for reflection coefficient, and the proof can be found    in  \cite{RN20}.
\begin{proposition}
	Suppose that $q(x)\mp1\in H^{2,2}(\mathbb{R}^\pm)$, then the reflection coefficient satisfies
	\begin{equation}
		\|\log(1-|r|^2)\|_{L^{p}(\mathbb{R})}<\infty \text{ for any } p\geq1.
	\end{equation}
\end{proposition}

Now we introduce the discrete spectrum. Suppose that $a(z)$ has finite $N$ simple zeros $ z_k$ $(k=0,\cdots, N-1)$ on $D_1\doteq\{z\in\mathbb{C}^+:\Im z>0,\Re z\geq0\}$, then symmetries (\ref{symmetry22}) implies 
that the discrete spectrum appears in double pairs
$\mathcal{Z}=\{z_k,\bar{z}_k,-\bar{z}_k,-z_k\}_{k=0}^{N-1}.$ When $z_k=-\bar{z}_k$,
the corresponding quaternate zeros $\{z_k,\bar{z}_k,-\bar{z}_k,-z_k\}$ degenerate into a pair  $\{-i, i\}$, which  corresponds  to solitons.
We further classify these  discrete spectrum as
\begin{equation}
	\mathcal{Z}_1\doteq\{z_k\}_{k=0}^{N-1}, \ \ \mathcal{Z}_2\doteq\{-\bar{z}_k\}_{k=0}^{N-1}, \ \  \mathcal{Z}^+\doteq\mathcal{Z}_1\cup\mathcal{Z}_2, \ \ \mathcal{Z}\doteq\mathcal{Z}^+\cup\overline{\mathcal{Z}^+}, \label{spcset}
\end{equation}
where   $\overline{\mathcal{Z}^+}$ formed by the complex conjugates of $\mathcal{Z}^+$.  The discrete spectrum can be seen in Figure  $\ref{discrete spectrum}$.

\begin{figure}
	\centering
{\includegraphics[width=0.35\linewidth]{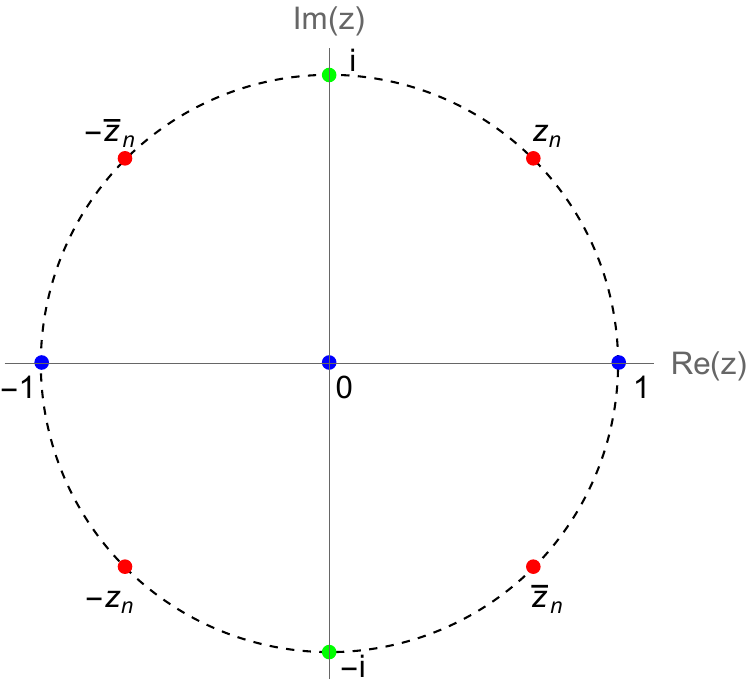}}
	 \caption{\small Distribution of   the discrete spectrum on circle  $|z|=1$: the jump contour $ \mathbb{R} $,   the red dots  (\textcolor{red}{$\centerdot$ }) and green dots (\textcolor{green}{$\centerdot$ }) represent the zeros of $a(z)$,
		they are corresponding  to breathers and solitons respectively. The blue  dots  (\textcolor{blue}{$\centerdot$}) represent singularities.  }\label{discrete spectrum}
	\end{figure}
Next we will state some properties of the zeros of $a(z)$.
From the symmetries of $\psi^{\pm}(z)$ we know that there exists a constant $\gamma_k\in\mathbb{C}$ such that
\begin{equation}\label{gamma}
	\psi_1^+(z_k)=\gamma_k\psi_2^-(z_k),\hspace{0.5cm}\psi_2^+(z_k)=\bar{\gamma}_k\psi_1^-(z_k),\hspace{0.5cm}\psi_1^+(z_k)=\gamma_k\psi_2^-(-\bar{z}_k),
\end{equation}
where  $\gamma_k$ is  known as the connection coefficient related to $z_k$.
\begin{proposition}
	Suppose $q(x)\mp1\in L^{1,2}(\mathbb{R}^\pm)$, then the discrete eigenvalues are simple, finite and distribute on the circle $|z|=1$.
\end{proposition}
\begin{proof}
	Let $z_k\in D_1$ be a zero of $a(z)$.
	It is well known that the Dirac operator for defocusing case is self-adjoint, meaning that the spectral parameter $\lambda(z_k)$ is real, which leads to the fact that
	\begin{equation}
		\Im\lambda(z_k)=\frac{|z_k|^2-1}{2|z_k|^2}\Im z_k=0.\label{iml}
	\end{equation}
Therefore,  all the zeros of $a(z)$ are distributed on the real axis and the unit circle. Next we prove the zeros are not real. 
First we note that $z_k\in\Real\setminus\{0,\pm 1\}$ cannot be a zero of $a(z)$, as it would contradict \eqref{a>1}. We will then show that $0$ and $\pm 1$ cannot be zeros either. 
The functions
	\begin{equation}
		f_k(\theta)=\Big|\partial_{\theta}^k\det[\psi_1^+(\e^{i\theta}),\psi_2^-(\e^{i\theta})]\Big|,\hspace{0.5cm}k=0,1
	\end{equation}
	are continuous for $\theta\in[0,\pi/2]$. Assume that $z=1$ is an accumulation point, according to the Bolzano-Weierstrass theorem, there exist sequences $\theta_j^{(k)}$, $k=0,1$, satisfying $\lim_{j\rightarrow\infty}\theta_j^{(k)}=0$ and $f_k(\theta_j^{(k)})=0$ for each $j$. From \eqref{scatteringdata1}, $a(z)=o(1)$ as $z\to1$. This just contradicts the fact in (\ref{a>1}). The proof is the same when $z=-1$ is an accumulation point. 
	In brief, all the zeros of $a(z)$ are distributed on the unit circle.
	
	From Proposition \ref{symmetries}, we have the symmetries of $\psi^{\pm}$ at $z_k$:
	\begin{align*}
		\bar{\psi}_1^+(\bar{z}_k^{-1})=\bar{\gamma}_k\bar{\psi}_2^-(\bar{z}_k^{-1}), \quad \sigma_1\psi_2^+(z_k^{-1})=\bar{\gamma}_k\sigma_1\psi_1^-(z_k^{-1}),\quad \psi_1^+(z_k)=\bar{\gamma}_k\psi_2^-(z_k),
	\end{align*}
	which imply that $\gamma_k=\bar{\gamma}_k$, i.e., $\gamma_k\in\mathbb{R}$.
	The condition $q(x)\mp1\in L^{1,1}(\mathbb{R}^{\pm})$ guarantees the existence of  $\frac{\partial a}{\partial\lambda}$. Evaluating the differentiation of \eqref{scatteringdataexpression} at $z_k$, we have
	\begin{equation}
		\frac{\partial a}{\partial\lambda}\Big|_{z=z_k}={\frac{\det[\partial_{\lambda}\psi_1^+,\psi_2^-]+\det[\psi_1^+,\partial_{\lambda}\psi_2^-]}{1-z^{-2}}}\Big|_{z=z_k}.\label{p1}
	\end{equation}
	Differentiating the two functions appear in the numerator of \eqref{p1} with respect to $x$ and using the scattering problem \eqref{lax0} we obtain that
	\begin{align*}
		\frac{\partial}{\partial x}\det[\partial_\lambda\psi_1^+,\psi_2^-]=\det[X_{\lambda}\psi_1^+,\psi_2^-]+\det[X\partial_{\lambda}\psi_1^+,\psi_2^-]+\det[\partial_{\lambda}\psi_1^+,X\psi_2^-]=\i\det[\sigma_3\psi_1^+,\psi_2^-],
	\end{align*}
	and
	\begin{align*}
		\frac{\partial}{\partial x}\det[\psi_1^+,\partial_\lambda\psi_2^-]=\det[\psi_1^+,X_{\lambda}\psi_2^-]+\det[X\psi_1^+,\partial_{\lambda}\psi_2^-]+\det[\psi_1^+,X\partial_{\lambda}\psi_2^-]=\i\det[\psi_1^+,\sigma_3\psi_2^-].
	\end{align*}
	According to (\ref{gamma}) and the decaying properties for each column at $z_k$, we can easily get that
	\begin{align*}
		&\det[\partial_{\lambda}\psi_1^+,\psi_2^-]=\i\gamma_k\int_{-\infty}^x\det[\sigma_3\psi_2^-(z_k,s),\psi_2^-(z_k,s)]ds,\\
		&\det[\psi_1^+,\partial_{\lambda}\psi_2^-]=\i\gamma_k\int_{x}^{\infty}\det[\sigma_3\psi_2^-(z_k,s),\psi_2^-(z_k,s)]ds.
	\end{align*}
	The symmetries of $\psi^\pm$ give us  $\psi_2^-(z_k)=-\i z_k^{-1}\sigma_1\overline{\psi_2^-(z_k)}$, which implies
	\begin{equation}
		\frac{\partial a}{\partial\lambda}\Big|_{z=z_k}=\frac{\gamma_k}{2\zeta(z_k)}\int_{\mathbb{R}}|\psi_2^-(z_k)|^2ds.
	\end{equation}
	Since  $\gamma_k$ is real and $\zeta(z_k)$ is imaginary, we conclude that $\frac{\partial a}{\partial\lambda}\big|_{z=z_k}\in \i\mathbb{R}$. It follows that the zeros of $a(z)$ are simple.
\end{proof}
Now, we can define the trace formula
\begin{equation}\label{traceformula}
	a(z)=\left(\frac{z-\i}{z+\i}\right)\prod_{n=0}^{N-1}\frac{(z-z_n)(z+\bar{z}_n)}{(z-\bar{z}_n)(z+z_n)}\exp\left(-\frac{1}{2\pi \i}\int_{\mathbb{R}}\frac{\log(1-|r(s)|^2)}{s-z}\d s\right)\,.
\end{equation}
\begin{remark}
In the trace formula, we claim that $\pm\i$ are always present for this paper. 
To show this, consider the function $a(z)=\left(\frac{z-\i}{z+\i}\right)^\delta\prod_{n=0}^{N-1}\frac{(z-z_n)(z+\bar{z}_n)}{(z-\bar{z}_n)(z+z_n)}\exp\left(-\frac{1}{2\pi \i}\int_{\mathbb{R}}\frac{\log(1-|r(s)|^2)}{s-z}\d s\right)$, where $\delta=1$ indicates that $\pm\i$ are zeros of $a(z)$ and $\delta=0$ indicates otherwise. From the asymptotic behavior \eqref{asymptoticsfora2}, we have that $\lim_{z\to 0} a(z)=-1$, which leads to
\begin{equation*}
a(0)=-1=(-1)^\delta\exp\left(-\frac{1}{2\pi \i}\int_{\mathbb{R}}\frac{\log(1-|r(s)|^2)}{s}\d s\right)\,.
\end{equation*}
Using the symmetries \eqref{symmetry22}, which implies $|r(z)|^2=|r(-z)|^2$, we have that
\begin{equation*}
\int_{-\infty}^0\frac{\log(1-|r(s)|^2)}{s}\d s=\int_{+\infty}^0\frac{\log(1-|r(s)|^2)}{s}\d s\,.
\end{equation*}
This simplifies the equation for $a(0)$ to $-1=(-1)^\delta$.
Therefore, one has $\delta=1$, meaning that $\pm\i$ are indeed zeros of $a(z)$. However, $\pm\i$ are not always the zeros of $a(z)$ in the generic case. As shown in \cite{RN12} that if $q_+/q_-=1$, then $\pm\i$ are no longer zeros of $a(z)$.
\end{remark}

At the end of this section, we demonstrate the time evolution of the scattering data  $a(z)$,  $b(z)$ and $r(z)$. Define $\psi^\pm(z;x,t)\doteq\psi^\pm(z;x)H(z;t)$ as the the time-dependent eigenfunctions, where $H(z;t)$ is a function to be determined.
Substituting  $\psi^\pm(z;x,t)$  into   \eqref{lax0} leads to
\begin{equation}
	( \partial_t-T)[\psi^\pm(z;x )H(z;t)]=0. \label{t1}
\end{equation}
As $x\rightarrow\pm\infty$, (\ref{t1}) gives $H(z;t)=\e^{\i\zeta(z)(4\lambda^2+2)t\sigma_3}.$
Differentiating   (\ref{sydd}) with respect to $t$ and using (\ref{t1}),  we derive an ODE that the scattering matrix satisfies
\begin{equation}
	\partial_tS=\i\zeta(z)(4\lambda^2(z)+2)[\sigma_3,S],
\end{equation}
which indicates that
\begin{align}
	&a(z,t)=a(z,0),\quad b(z,t)=b(z,0)\e^{\i\zeta(4\lambda^2+2)t},\quad r(z,t)=r(z,0)\e^{\i\zeta(4\lambda^2+2)t};\nonumber\\
	&z_k(t)=z_k(0),\quad \gamma_k(t)=\gamma_k(0)\e^{i\zeta(4\lambda^2+2)t}.\nonumber
\end{align}
Hereafter, for convenience, we still use $a(z)$, $b(z)$ and $r(z)$ to denote $a(z,t)$, $b(z,t)$ and $r(z,t)$, respectively.

\subsection{An RH characterization of the mKdV equation}\label{subsec:RH formulation}
Now we introduce the sectionally meromorphic function
\begin{equation}
	m(z)=m(z;x,t)=\left\{\begin{aligned}
		&\left(\frac{\mu_1^+(z;x,t)}{a(z)}, \mu_2^-(z;x,t)\right),& &z\in\mathbb{C}^+,\\
		&\left(\mu_1^-(z;x,t), \frac{\mu_2^+(z;x,t)}{\overline{a(\bar{z})}}\right),& &z\in\mathbb{C}^-,
	\end{aligned}\right.
\end{equation}
then the following matrix RH problem is proposed.

\begin{RHP}\label{rhpm}
	Find a $2\times2$ matrix-valued function $m(z;x,t)$ such that
	\begin{enumerate}
		\item $m(z)$ is meromorphic for $z\in\mathbb{C}\setminus\mathbb{R}$, with poles belonging to the set $\mathcal{Z}$ defined in \eqref{spcset}.
		\item  $m(z)$ satisfies the following symmetries
		\begin{align}
			&m(z)=\sigma_1\overline{m(\bar{z})}\sigma_1 =z^{-1} m(z^{-1})\sigma_2=\overline{m(-\bar{z})}.\label{symmetrym}
		\end{align}
		\item $m(z)$ has the following asymptotics
		\begin{align*}
			m(z;x,t)=I+\mathcal{O}(z^{-1}),\quad z\to\infty;\qquad
			zm(z;x,t)=\sigma_2+\mathcal{O}(z),\quad z\to 0.
		\end{align*}
		\item $m_{\pm}(z;x,t)=\lim\limits_{\mathbb{C}^{\pm}\ni z'\rightarrow z}m(z';x,t)$ exist for any $z\in\mathbb{R}\setminus\{0\}$ and meet the jump relation $m_+(z;x,t)=m_-(z;x,t)V(z)$, where
		\begin{equation}
			V(z)=\begin{pmatrix}
				1-|r(z)|^2 & -\overline{r(z)}\e^{2\i t\theta(z)}\\
				r(z)\e^{-2\i t\theta(z)} & 1
			\end{pmatrix},
		\end{equation}
		and
		\begin{equation}
			\theta(z)=\zeta(z)\{x/t+4\lambda^2(z)+2\}.\label{theta}
		\end{equation}
		\item $m(z;x,t)$ has residue conditions at the simple poles in $\mathcal{Z}=\mathcal{Z}^+\cup\overline{\mathcal{Z}^+}$
		\begin{equation}\label{resm}
			\begin{split}
				\Res_{z=z_k}m(z;x,t)=\lim_{z\rightarrow z_k}m(z;x,t)\begin{pmatrix}
					0 & 0\\
					c_k(x,t) & 0
				\end{pmatrix},\quad
				\Res_{z=\bar{z}_k}m(z;x,t)=\lim_{z\rightarrow \bar{z}_k}m(z;x,t)\begin{pmatrix}
					0 & \bar{c}_k(x,t)\\
					0 & 0
				\end{pmatrix},\\
				\Res_{z=-\bar{z}_k}m(z;x,t)=\lim_{z\rightarrow -\bar{z}_k}m(z;x,t)\begin{pmatrix}
					 0 & 0\\
					-\bar{c}_k(x,t) & 0
			\end{pmatrix},\quad 
				\Res_{z=-z_k}m(z;x,t)=\lim_{z\rightarrow -z_k}m(z;x,t)\begin{pmatrix}
					0 & -c_k(x,t)\\
					0 & 0
				\end{pmatrix}.
			\end{split}
		\end{equation}
		where
		\be
				c_k(x,t)=\frac{\gamma_k(0)}{a'(z_k)}\e^{-2\i t\theta(z_k)}\doteq c_k\e^{-2\i t\theta(z_k)},\quad
				c_k=\frac{\gamma_k(0)}{a'(z_k)}=\frac{2z_k}{\int_{\mathbb{R}}|\psi_2^+(z_k;x,0)|^2dx}=z_k|c_k|.
		\ee
	\end{enumerate}
\end{RHP}
\begin{remark}
The solution of  RH Problem \ref{rhpm} preserves the symmetry from $z$ to $-z$, ensuring that the potential recovered from this solution is real, consistent with the fact that the solution of the mKdV equation is real.
\end{remark}
It then follows the reconstruction formula
\begin{equation}
	q(x,t)=\lim_{z\rightarrow\infty}\i zm_{21}(z;x,t).
\end{equation}


\subsection{Signature tables}\label{subsec:signature table}
Large-time asymptotic behavior of  RH problem  \ref{rhpm}  is influenced by  decay/growth of   oscillatory terms  $\e^{\pm 2\i t\theta(z)}$ and  phase points of  $\theta(z)$.  Let $\xi=x/t$, direct calculation  gives
$$  \theta'(z)=\frac{3}{2}z^2+\frac{\xi+3}{2z^2}+\frac{3}{2z^4}+\frac{\xi+3}{2}, $$
from which we can get six phase points. 
Moreover, he decay/growth of   oscillatory terms  $\e^{\pm 2\i t\theta(z)}$  is determined by  the sign of  $\Re(2\i t \theta(z))$.
The decaying region of  $\Re (2\i t \theta(z))$ is shown in Figure \ref{figtheta}. 
\begin{proposition}[Distribution of saddle points]\label{relation of xi}
    In addition to the two fixed saddle points $\pm\i$, there exist four saddle points satisfying the following properties for
    different $\xi$ (see Figure \ref{figsaddle}):\\
    - For $\xi<-6$, the four saddle points $\xi_j:=\xi_j(\xi)$, $j=1,2,3,4$ lie on the jump contour $\Sigma=\mathbb{R}\backslash\{0\}$.
        Moreover, we have $\xi_4<-1<\xi_3<0<\xi_2<1<\xi_1$ and $\xi_1=\frac{1}{\xi_2}=-\frac{1}{\xi_3}=-\xi_{4}$;\\
    - For $-6<\xi<6$, the four saddle points are away from the coordinate axes;\\
    - For $\xi>6$, the four saddle points lie on the imaginary axis with $\Im\xi_1>1>\Im\xi_2>0>\Im\xi_3>-1>\Im\xi_4$ and $\xi_1\xi_2=\xi_3\xi_4=-1$.
\end{proposition}
\begin{proof}
    Letting $\theta'(z)=0$ and factoring the left-hand side gives us
    \begin{equation}
       (1+z^2)\left(3z^4+\xi z^2+3\right)=0.
    \end{equation}
    Solving for $z$, we have $z=\pm\i$ and 
    \begin{equation}
        z^2=-\frac{\xi+\sqrt{\xi^2-36}}{6}, \quad {\rm or}\quad  z^2=-\frac{\xi-\sqrt{\xi^2-36}}{6}.
    \end{equation}
    For $\xi<-6$, we have $-\frac{\xi\pm\sqrt{\xi^2-36}}{6}>0$,
    and the four roots are as follows.
    \begin{align}
    \xi_1=\sqrt{-\frac{\xi-\sqrt{\xi^2-36}}{6}}, \quad \xi_4=-\sqrt{-\frac{\xi-\sqrt{\xi^2-36}}{6}}, \label{saddlexi1,4}\\
    \xi_2=\sqrt{-\frac{\xi+\sqrt{\xi^2-36}}{6}}, \quad \xi_3=-\sqrt{-\frac{\xi+\sqrt{\xi^2-36}}{6}}.\label{saddlexi2,3}
    \end{align}
    with  $\xi_4<-1<\xi_3<0<\xi_2<1<\xi_1$ and $\xi_1=\frac{1}{\xi_2}=-\frac{1}{\xi_3}=-\xi_{4}$.

    For $-6<\xi<6$, the discriminant $\xi^2-36$ is less than zero. Therefore, there exist four saddle points
    $\xi_j=\Re(\xi_j)+\i\Im(\xi_j)$, where $\Re(\xi_j), \Im(\xi_j)\neq 0$, $j=1,2,3,4$.

    For $\xi>6$, we have $-\frac{\xi\pm\sqrt{\xi^2-36}}{6}<0$, and the
    four pure imaginary saddle points are as follows
    \begin{align}
    \xi_1=\i\sqrt{\frac{\xi+\sqrt{\xi^2-36}}{6}}, \quad \xi_4=-\i\sqrt{\frac{\xi+\sqrt{\xi^2-36}}{6}}, \\
    \xi_2=\i\sqrt{\frac{\xi-\sqrt{\xi^2-36}}{6}}, \quad \xi_3=-\i\sqrt{\frac{\xi-\sqrt{\xi^2-36}}{6}},
    \end{align}
    with $\Im\xi_1>1>\Im\xi_2>0>\Im\xi_3>-1>\Im\xi_4$ and $\xi_1\xi_2=\xi_3\xi_4=-1$.
\end{proof}

According to the Figure \ref{figsaddle} and Figure \ref{figtheta}, it could be observed the following:
\begin{itemize}
\item For $\xi<-6$, there are four stationary phase points in addition to $i,-i$, all
of which are located in the jump contour as shown in Figure \ref{xiaoyu}, with the corresponding signature table in Figure \ref{figxi4sp}.
\item For $-6<\xi<-2$, The distribution of phase points is shown in Figure \ref{jieyu}, with the signature table shown in Figure \ref{figxi0sp}.
\item For $\xi>-2$, there exist again four stationary phase points besides $i,-i$.
    \begin{itemize}
    \item When $-2<\xi<6$, the four saddle points are away from the coordinate axis (both real and imaginary axis), corresponding to Figure \ref{jieyu}, with the signature table shown in Figure \ref{figxi00sp}. The asymptotic analysis for $-2<\xi<6$ could be seen as a specific case of the analysis for $-6<\xi<-2$.
    \item For $\xi>6$, the four saddle points are all distributed on the imaginary axis as shown in Figure \ref{dayu}, and the signature table is still shown in Figure \ref{figxi00sp}.
    \end{itemize}
\end{itemize}

\begin{figure}[htbp]
	\centering
	\subfloat[]{\includegraphics[width=0.25\linewidth]{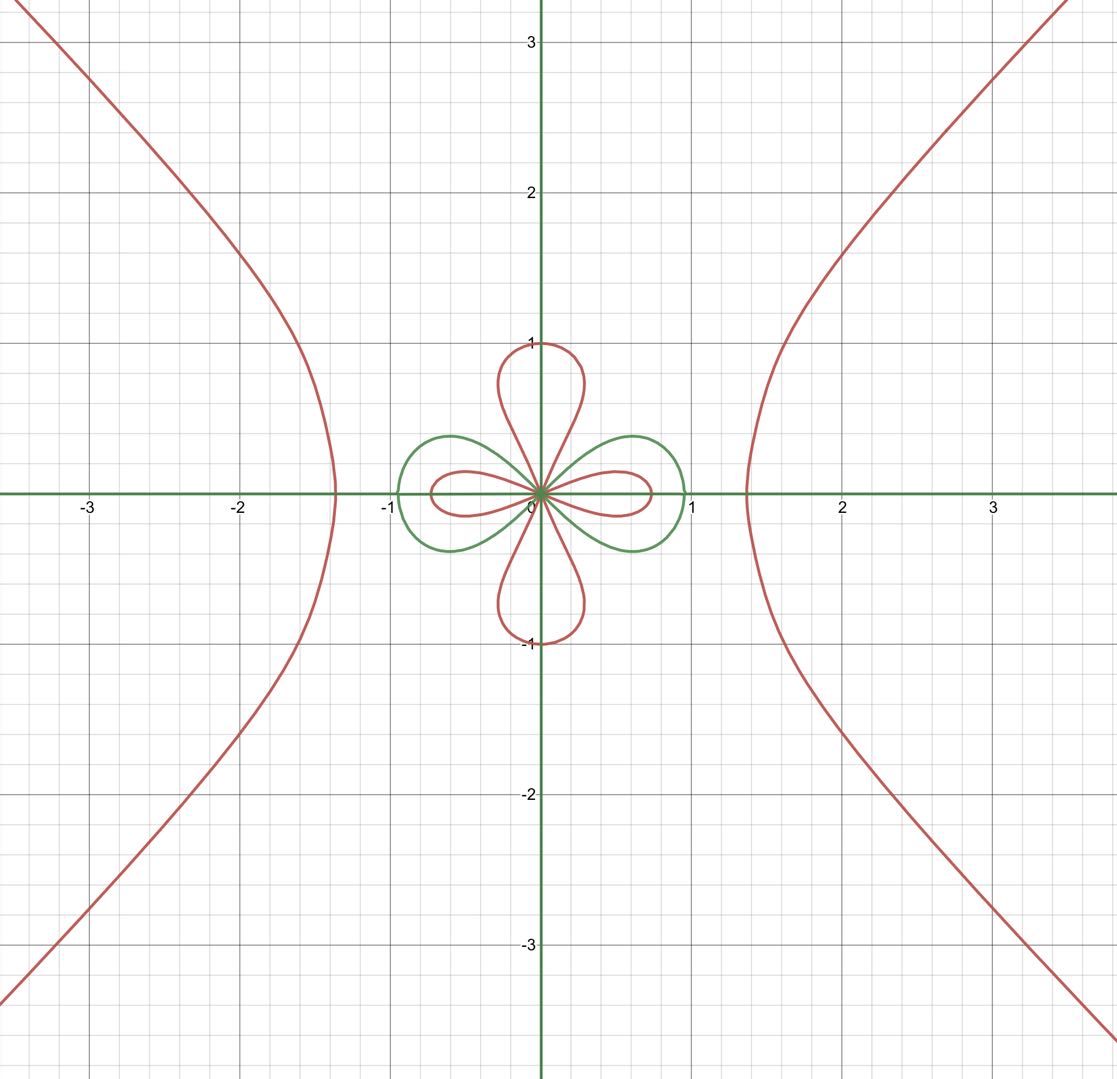}\hspace{0.5cm}\
		\label{xiaoyu}}
	\subfloat[]{\includegraphics[width=0.25\linewidth]{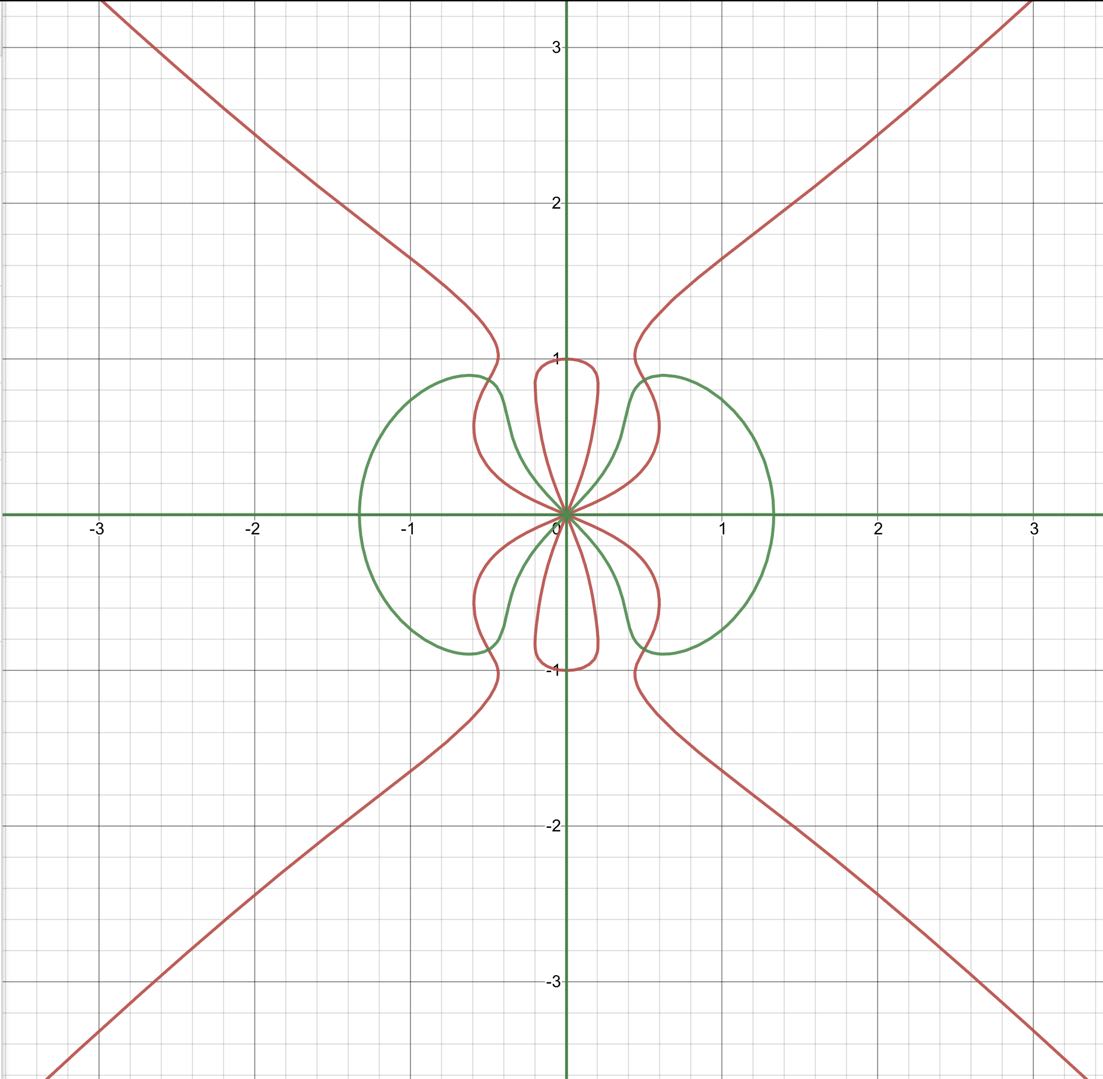}\hspace{0.5cm}
		\label{jieyu}}	
	\subfloat[]{\includegraphics[width=0.25\linewidth]{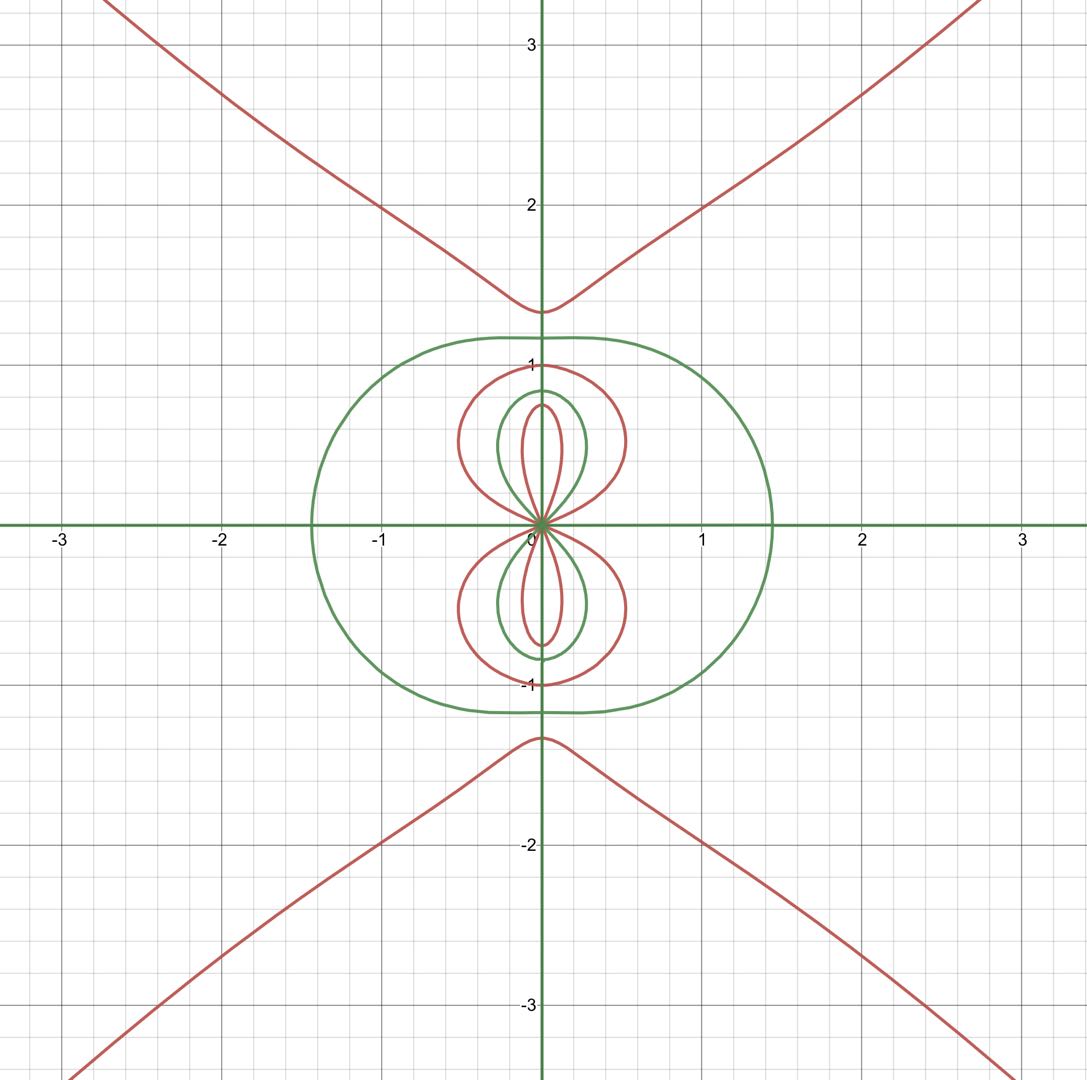}
		\label{dayu}}
	\caption{\footnotesize  Plots of the distributions for saddle points:
    $\textbf{(a)}$ $\xi<-6$,
    $\textbf{(b)}$ $-6<\xi<6$,
    $\textbf{(c)}$ $\xi>6$. The red curve represents $\Re \theta'(z)=0$, and the green curve represents $\Im \theta'(z)=0$. The intersection
    points are the saddle points which represent the zeros of $\theta'(z)=0$.}
	\label{figsaddle}
\end{figure}

\begin{figure}[htbp]
	\centering
	\subfloat[]{\includegraphics[width=0.26\linewidth]{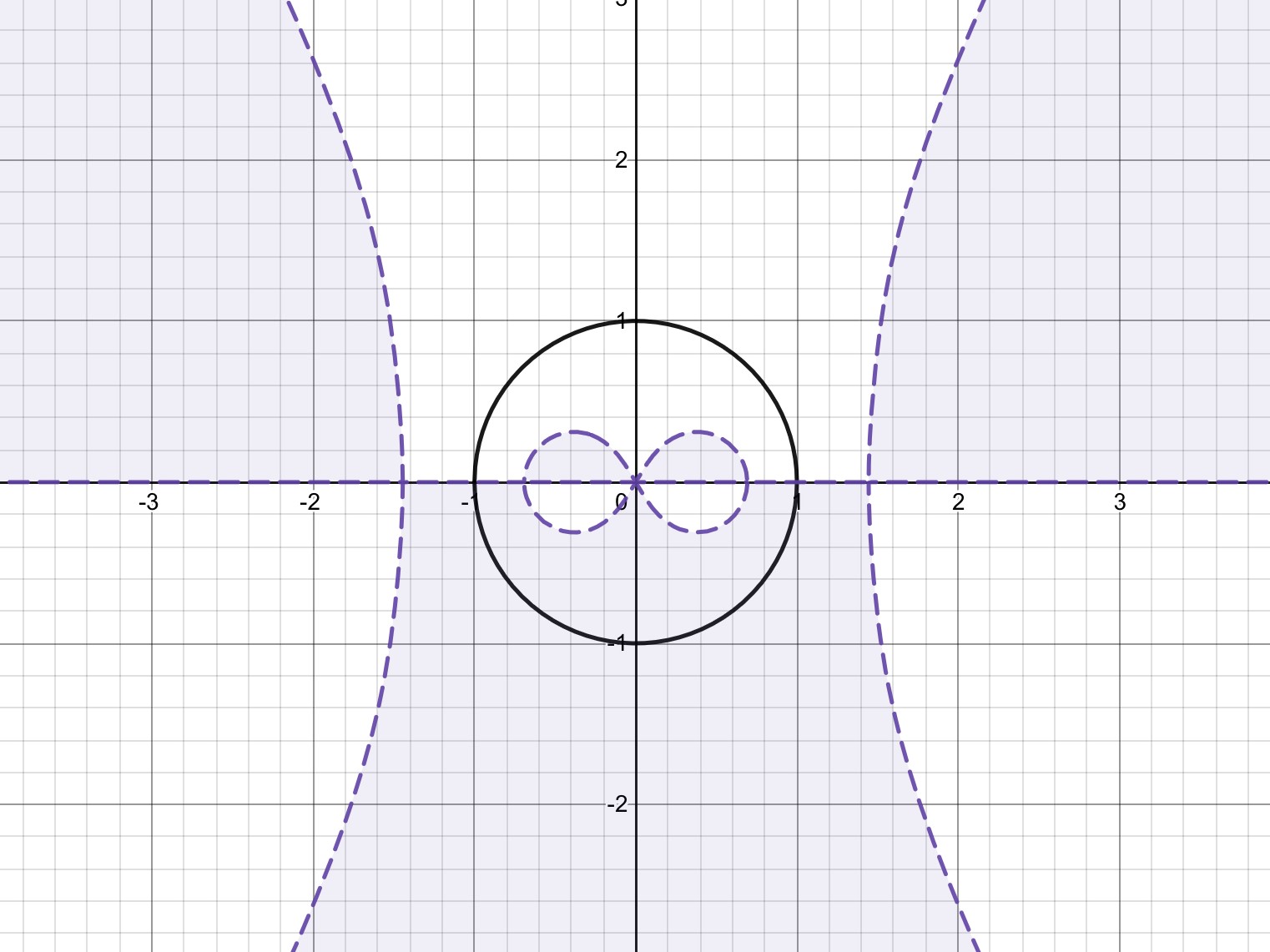}\hspace{0.5cm}
		\label{figxi4sp}}
	\subfloat[]{\includegraphics[width=0.26\linewidth]{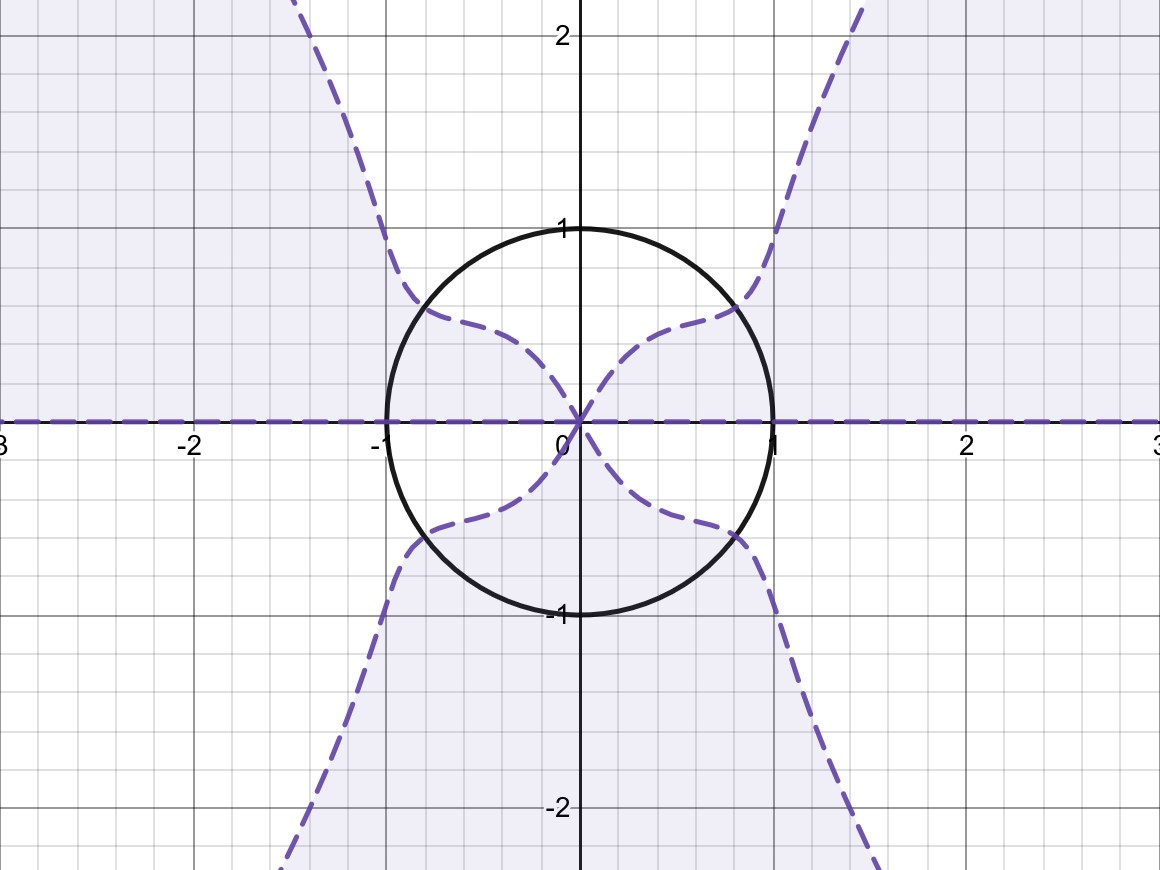}\hspace{0.5cm}
		\label{figxi0sp}}	
	\subfloat[]{\includegraphics[width=0.26\linewidth]{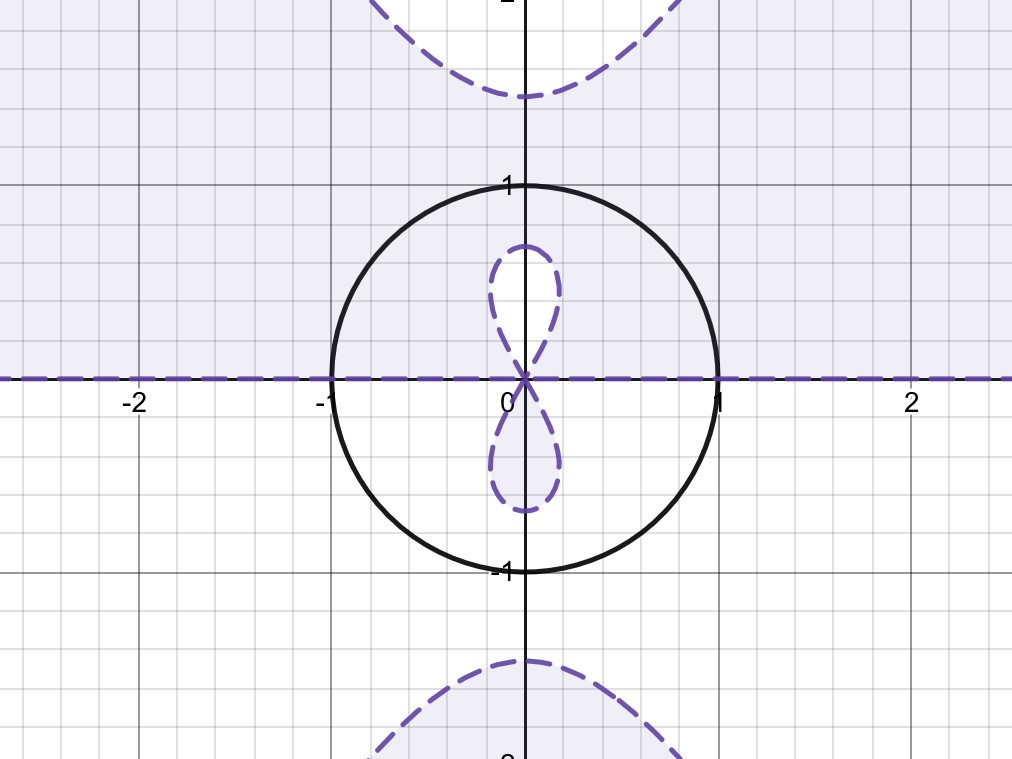}
		\label{figxi00sp}}
	\caption{\footnotesize  Signature table of $ \Re (2\i t \theta(z))$   with different $\xi$:
		$\textbf{(a)}$ $\xi<-6$,
		$\textbf{(b)}$ $-6<\xi<-2$,
		$\textbf{(c)}$ $\xi>-2$. In the purple region, $\Re(2\i t\theta)<0$, while in the white region, $\Re(2\i t\theta)>0$.
		The purple dashed curves are the critical curves.    }
	\label{figtheta}
\end{figure}

\section{Deformation of the RH Problem}\label{sec:deform}
In this section, two crucial steps are given to use the steepest descent method:
(i) converting the residue conditions of those poles, which are far away from the 
critical lines, into jump conditions on new auxiliary contours;
(ii) using a well-known factorization for the jump matrix to deform the jump line $\mathbb{R}$ 
into such lines that the jump matrices on them decay exponentially as $t\to\infty$.

\subsection{Interpolation and conjugation}
By simple calculation, we find that on the unit circle the phase function satisfies
\begin{equation}
	\Re(2\i t\theta(z))=-2t\sin\omega[\xi+2+4\cos^2\omega],
\end{equation}
where $\omega$ is the argument of $z=\e^{\i\omega}$.

For $-6<\xi<-2$,  let $\xi_0=\sqrt{-\frac{\xi+2}{4}}$,  then the discrete spectrum $\mathcal{Z}_1$ defined in (\ref{spcset}) decomposes into three sets: points with  $\Re(z_k)>\xi_0$ and exponentially decaying connection coefficients $c_k(x,t)=c_k\e^{-2\i t\theta(z_k)}$ as $t\to\infty$; points with $0\leq \Re(z_k)<\xi_0$ and growing connection coefficients; and a single point with $\Re(z_k)=\xi_0$ and a bounded connection coefficient in $t$ (see Figure \ref{figxi0sp}). \br Let \er $\rho>0$ sufficiently small such that
\begin{equation}
	\rho<\frac{1}{2}\min\left\{\min\limits_{z_j,z_k\in \mathcal{Z}^+}|\Re(z_j-z_k)|, \min\limits_{z_k\in \mathcal{Z}^+}|\Im(z_k)|\right\}.\nonumber
\end{equation}
We divide the index set $H\doteq\{0,1,\cdots,N-1\}$ into two subsets
\begin{align}
	&\triangle=\left\{j \in H:\Re(z_j)>\xi_0 \right\},\hspace{0.5cm}\nabla=\left\{j\in H:  0\leq \Re(z_j)\leq\xi_0\right\},\nonumber
\end{align}
and define
\begin{align}\label{def: Lambda}
	&\Lambda=\left\{j \in H: | \Re(z_{j})-\xi_0 |<\rho \ {\rm and } \ |\Re(\bar z_{j})+\xi_0 |<\rho\right\}, 
\end{align}
then the sets $|\Re(z-z_j)|\leq \rho$ are pairwise disjoint. If $j_0\in\Lambda\neq\varnothing$, then we have $|\e^{ \i t\theta(z_{j_0})}|=\mathcal{O}(1)$.

Note that we can use a well-known factorization to deform the jump matrix $V$:
\begin{equation}\label{V}
	V(z)=\begin{pmatrix}
		1-|r(z)|^2 & -\overline{r(z)}e^{2\i t\theta(z)}\\
		r(z)e^{-2\i t\theta(z)} & 1
	\end{pmatrix}=B(z)T_0(z)B(z)^{-\dag},
\end{equation}
where
\begin{equation}
	B(z)=\begin{pmatrix}
		1 & 0\\
		\frac{r(z)}{1-|r(z)|^2}e^{-2\i t\theta} & 1
	\end{pmatrix},\quad T_0(z)=(1-|r|^2)^{\sigma_3},\quad B(z)^{-\dag}=\begin{pmatrix}
		1 & \frac{-\bar{r}(z)}{1-|r(z)|^2}e^{2\i t\theta}\\
		0 & 1
	\end{pmatrix},\nonumber
\end{equation}
and  $B^{\dag}$ is the Hermitian conjugate of $B$. 
Therefore, one can  extend $B(z)$ into $\mathbb{C}^-$ and $B^{-\dag}$ into $\mathbb{C}^+$, while the second term $T_0$ remains on $\mathbb{R}$. This deformation is helpful when the factors into regions in which the corresponding off-diagonal exponential terms $e^{\pm2\i t\theta}$ are decaying. We then start to handle the second term $T_0$.

Define the function
\begin{equation}\label{T}
	T(z;\xi)=-\prod_{k\in\triangle}\frac{(z-z_k)(z+\bar{z}_k)}{(zz_k-1)(z\bar{z}_k+1)}\exp\left(-\frac{1}{2\pi \i}\int_{\mathbb{R}}\log(1-|r(s)|^2)\left(\frac{1}{s-z}-\frac{1}{2s}\right)ds\right).
\end{equation}

\begin{proposition}\label{propertiesofT}
	$T(z;\xi)$ is a meromorphic function in $\mathbb{C}\setminus\mathbb{R}$ and has the following properties:
	\begin{itemize}
		\item $T(z)$ has simple poles at $\pm z_k$ and simple zeros at $\pm \bar{z}_k$ with $\Re(z_k) > \xi_0$.
		\item $T(z)$ satisfies the jump condition
		\begin{equation}\label{jumpofT}
			T_-(z;\xi)=T_+(z;\xi)(1-|r(z)|^2),\hspace{0.5cm}z\in\mathbb{R}.
		\end{equation}
		\item $T(z)$ has the following symmetries
		\begin{equation}\label{symmetriesT}
			\overline{T(\bar{z};\xi)}=T^{-1}(z;\xi)=T(z^{-1};\xi)=T(-z;\xi).
		\end{equation}
		\item $T(z)$ has the following asymptotics:
		\begin{equation}
			T(\infty;\xi)=\lim_{z\rightarrow\infty}T(z;\xi)=(-1)^{|\triangle|},\quad z\rightarrow\infty,\label{Tinfty}
		\end{equation}
		where $|\triangle|$ is the cardinality of the set $\triangle$, and $|T(\infty,\xi)|=1$.
		\item As $z\rightarrow\infty$, we have the asymptotic expansion:
		\begin{equation}\label{expansionT}
			T(z;\xi)=T(\infty;\xi)\left(I-\frac{1}{z}\left(\sum_{k\in\triangle}4\i\Im(z_k)-\frac{1}{2\pi \i}\int_{\mathbb{R}}\log(1-|r(s)|^2)ds\right)+o(z^{-1})\right).
		\end{equation}
		\item $\frac{a(z)}{T(z;\xi)}$ is holomorphic in $\mathbb{C}^+$ and there exists a constant $C(q_0)$ so that
		\begin{equation}
			\left|\frac{a(z)}{T(z;\xi)}\right|<C(q_0),\quad z\in \mathbb{C}^+.
		\end{equation}
		Moreover, $\frac{a(z)}{T(z;\xi)}$ can be continuously extended to $\mathbb{R}$, with $|\frac{a(z)}{T(z;\xi)}|=1$ for $z\in\mathbb{R}$.
	\end{itemize}
\end{proposition}
\begin{proof}
	The first three properties are easy to prove according to the definition of $T(z)$, we begin from the fourth.
	As $z\rightarrow\infty$, the product
	\begin{align*}
		\prod_{k\in \triangle}\frac{(z-z_k)(z+\bar{z}_k)}{(zz_k-1)(z\bar{z}_k+1)}=\prod_{k\in \triangle}\frac{1}{z_k}\left(\frac{z-z_k}{z-\bar{z}_k}\right)\frac{-1}{\bar{z}_k}\left(\frac{z+\bar{z}_k}{z+z_k}\right)
		\rightarrow(-1)^{|\triangle|},
	\end{align*}
	together with  $\frac{1}{s-z}-\frac{1}{2s}\rightarrow-\frac{1}{2s}$ and $\int_{\mathbb{R}}\frac{\log(1-|r(s)|^2)}{2s}ds=0$ implies the fourth property.
	The next property is a simple corollary of the last one.
	In the end, from the trace formula (\ref{traceformula}) we have
	\begin{equation}\label{ratio}
		\frac{a(z)}{T(z)}=-\prod_{k\in\nabla}\frac{(z-z_k)(z+\bar{z}_k)}{(z-\bar{z}_k)(z+z_k)}.
	\end{equation}
Note that the absolute value in the right-hand-side is not larger than $1$ for $z\in\mathbb{C}^+$. So we have proved the proposition.
\end{proof}
Next, we proceed with interpolations and conjugations.
We first introduce the interpolation function $G(z)$:
\begin{itemize}\label{G}
	\vspace{-1em}
\item For $j\in\triangle\setminus\Lambda$,
\begin{equation}
G(z)=\left\{\begin{aligned}
	&\left(\begin{array}{cc}
		1 & -\frac{(z-z_j)\e^{2\i t\theta(z_j)}}{c_j}\\
		0 & 1
	\end{array}\right), &&|z-z_j|<\rho,\\
	&\left(\begin{array}{cc}
		1 & \frac{(z+\bar{z}_j)\e^{2\i t\theta(z_j)}}{\bar{c}_j}\\
		0 & 1
	\end{array}\right),&& |z+\bar{z}_j|<\rho,\\
	&\left(\begin{array}{cc}
		1 & 0\\
		-\frac{(z-\bar{z}_j)\e^{-2\i t\theta(\bar{z}_j)}}{\bar{c}_j} & 1
	\end{array}\right),&&|z-\bar{z}_j|<\rho,\\
	&\left(\begin{array}{cc}
		1 & 0\\
		\frac{(z+z_j)\e^{-2\i t\theta(\bar{z}_j)}}{c_j} & 1
	\end{array}\right),&& |z+z_j|<\rho;
\end{aligned}\right.
\end{equation}
\item For $j\in\nabla\setminus\Lambda$,
\begin{equation}
G(z)=\left\{\begin{aligned}
	&\left(\begin{array}{cc}
		1 & 0\\
		-\frac{c_j\e^{-2\i t\theta(z_j)}}{z-z_j} & 1
	\end{array}\right),&&|z-z_j|<\rho,\\
	&\left(\begin{array}{cc}
		1 & 0\\
		\frac{\bar{c}_j\e^{-2\i t\theta(z_j)}}{z+\bar{z}_j} & 1
	\end{array}\right),&&|z+\bar{z}_j|<\rho,\\
	&\left(\begin{array}{cc}
		1 & -\frac{\bar{c}_j\e^{2\i t\theta(\bar{z}_j)}}{z-\bar{z}_j}\\
		0 & 1
	\end{array}\right),&&|z-\bar{z}_j|<\rho,\\
	&\left(\begin{array}{cc}
		1 & \frac{c_j\e^{2\i t\theta(\bar{z}_j)}}{z+z_j}\\
		0 & 1
	\end{array}\right),&&|z+z_j|<\rho;
\end{aligned}\right.
\end{equation}
\item Elsewhere, $G(z)=I$.
\end{itemize}

Then we can introduce the first transformation which converts the poles that are far away from the critical lines into jumps
\begin{equation}\label{m1}
	m^{(1)}(z)=T(\infty)^{-\sigma_3}m(z)G(z)T(z)^{\sigma_3}.
\end{equation}
Define the contour
\begin{equation}\label{contour: Sigma1}
	\Sigma^{(1)}\doteq\mathbb{R}\cup\left\{\bigcup_{j\in H\setminus\Lambda}\left\{z\in\mathbb{C}:|z\pm z_j|=\rho \text{ or }|z\pm\bar{z}_j|=\rho\right\}\right\}.
\end{equation}
As shown in Figure \ref{Sigma1}, the small circles around the poles are oriented counterclockwise in $\mathbb{C}^+$ and clockwise in $\mathbb{C}^-$.
Then  $m^{(1)}(z)$   satisfies the following  RH  problem.
\begin{figure}
	\centering
	{\includegraphics[width=0.35\linewidth]{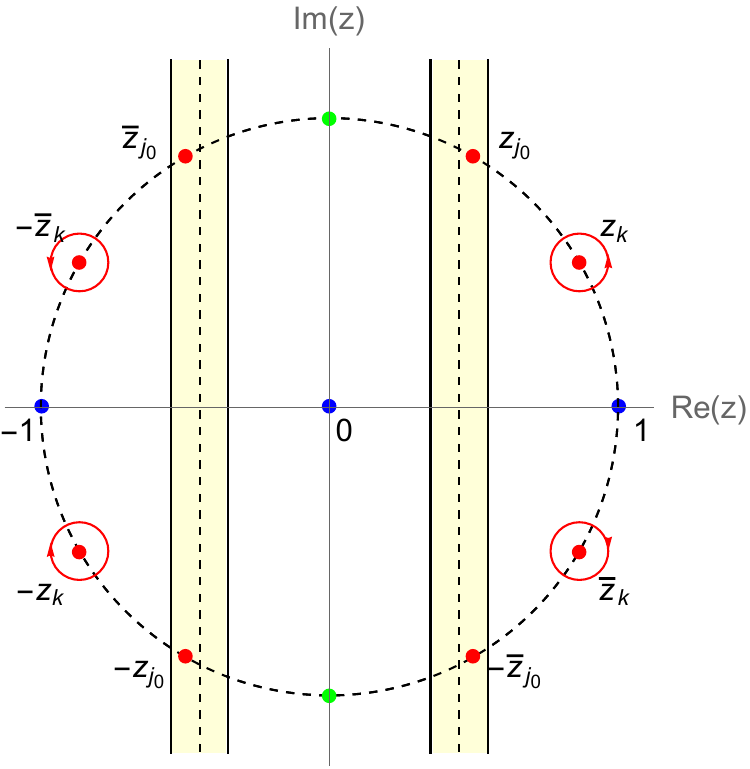}}
\caption{The dashed lines are where $\Re z=\pm\xi_0$. We divide the discrete spectrum in the first quadrant $D_1$  into three sets: $\triangle\setminus\Lambda$, $\nabla\setminus\Lambda$ and $\Lambda$, with the poles reserved in $\Lambda$. By the symmetries, the discrete spectrum in the second quadrant is also divided into three sets. }\label{Sigma1}
\end{figure}

\begin{RHP}\label{rhp1}
	Find a $2\times2$ matrix-valued function $m^{(1)}(z;x,t)$ such that
	\begin{enumerate}
		\item $m^{(1)}(z;x,t)$ is meromorphic in $\mathbb{C}\setminus\Sigma^{(1)}$, where $\Sigma^{(1)}$ is defined in \eqref{contour: Sigma1}.
		\item $m^{(1)}$ has the following asymptotics
		\begin{align}
			&m^{(1)}(z;x,t)=I+\mathcal{O}(z^{-1}),\text{ as }z\rightarrow\infty,\nonumber\\
			&zm^{(1)}(z;x,t)=\sigma_2+\mathcal{O}(z),\text{ as }z\rightarrow 0. \nonumber
		\end{align}
		\item $m^{(1)}_{\pm}(z;x,t)$ exist for $z\in\Sigma^{(1)}$ and meet the jump relation $m_+^{(1)}(z;x,t)=m_-^{(1)}(z;x,t)V^{(1)}(z)$, where
		\begin{itemize}
			\item for $z\in\mathbb{R}$,
			\begin{equation}
				V^{(1)}(z)= \begin{pmatrix}
					1 & 0\\
					\frac{r(z)}{1-|r|^2}T_-^{-2}(z)\e^{-2\i t\theta} & 1
				\end{pmatrix}
			\begin{pmatrix}
					1 & \frac{-\bar{r}(z)}{1-|r|^2}T_+^{-2}(z)\e^{2\i t\theta}\\
					0 & 1
				\end{pmatrix},
			\end{equation}
			\item for $j\in\triangle\setminus\Lambda$,
			\begin{equation}
				V^{(1)}(z)=\left\{\begin{aligned}
					&\begin{pmatrix}
						1 & -\frac{(z-z_j)\e^{2\i t\theta(z_j)}}{c_j}T^{-2}(z)\\
						0 & 1
					\end{pmatrix},&&|z-z_j|=\rho, \\
					&\begin{pmatrix}
						1 & \frac{(z+\bar{z}_j)\e^{2\i t\theta(z_j)}}{\bar{c}_j}T^{-2}(z)\\
						0 & 1
					\end{pmatrix},& &|z+\bar{z}_j|=\rho,\\
					&\begin{pmatrix}
						1 & 0\\
						-\frac{(z-\bar{z}_j)\e^{-2\i t\theta(\bar{z}_j)}}{\bar{c}_j}T^2(z) & 1
					\end{pmatrix}, & &|z-\bar{z}_j|=\rho, \\
					&\begin{pmatrix}
						1 & 0\\
						\frac{(z+z_j)\e^{-2\i t\theta(\bar{z}_j)}}{c_j}T^2(z) & 1
					\end{pmatrix}, & &|z+z_j|=\rho,
				\end{aligned}\right.
			\end{equation}
			\item for $j\in\nabla\setminus\Lambda$,
			\begin{equation}
				V^{(1)}(z)=\left\{\begin{aligned}
					&\begin{pmatrix}
						1 & 0\\
						-\frac{c_j\e^{-2\i t\theta(z_j)}}{z-z_j}T^2(z) & 1
					\end{pmatrix}, & &|z-z_j|=\rho, \\
					&\begin{pmatrix}
						1 & 0\\
						\frac{\bar{c}_j\e^{-2\i t\theta(z_j)}}{z+\bar{z}_j}T^2(z) & 1
					\end{pmatrix}, & &|z+\bar{z}_j|=\rho, \\
					&\begin{pmatrix}
						1 & -\frac{\bar{c}_j\e^{2\i t\theta(\bar{z}_j)}}{z-\bar{z}_j}T^{-2}(z)\\
						0 & 1
					\end{pmatrix}, & &|z-\bar{z}_j|=\rho,\\
					&\begin{pmatrix}
						1 & \frac{c_j\e^{2\i t\theta(\bar{z}_j)}}{z+z_j}T^{-2}(z)\\
						0 & 1
					\end{pmatrix}, & &|z+z_j|=\rho.
				\end{aligned}\right.
			\end{equation}
		\end{itemize}
		
		\item If there exists a $j_0\in\Lambda$,  then $m^{(1)}(z;x,t)$ satisfies the following residue conditions at $\pm z_{j_0}$ and $\pm\bar{z}_{j_0}$:
		
		If $j_0\in\triangle\cap\Lambda$,
		\begin{align}
			&\Res_{z=z_{j_0}}m^{(1)}(z)=\lim_{z\rightarrow z_{j_0}}m^{(1)}(z)\begin{pmatrix}
				0 & c_{j_0}^{-1}\e^{2\i t\theta(z_{j_0})}T'(z_{j_0})^{-2} \\
				0 & 0
			\end{pmatrix},\nonumber\\
			&\Res_{z=\bar{z}_{j_0}}m^{(1)}(z)=\lim_{z\rightarrow\bar{z}_{j_0}}m^{(1)}(z)\begin{pmatrix}
				0 & 0\\
				\bar{c}_{j_0}^{-1}\e^{2\i t\theta(z_{j_0})}\bar{T'}(z_{j_0})^{-2} & 0
			\end{pmatrix},\nonumber\\
			&\Res_{z=-\bar{z}_{j_0}}m^{(1)}(z)=\lim_{z\rightarrow-\bar{z}_{j_0}}m^{(1)}(z)\begin{pmatrix}
				0 & -\bar{c}_{j_0}^{-1}\e^{2\i t\theta(z_{j_0})}\bar{T}'(z_{j_0})^{-2}\\
				0 & 0
			\end{pmatrix},\nonumber\\
			&\Res_{z=-z_{j_0}}m^{(1)}(z)=\lim_{z\rightarrow-z_{j_0}}m^{(1)}(z)\begin{pmatrix}
				0 & 0\\
				-c_{j_0}^{-1}\e^{2\i t\theta(z_{j_0})}T'(z_{j_0})^{-2} & 0
			\end{pmatrix}.\nonumber
		\end{align}
	 If $j_0\in\nabla\cap\Lambda$,
	\begin{align}
		&\Res_{z=z_{j_0}}m^{(1)}(z)=\lim_{z\rightarrow z_{j_0}}m^{(1)}(z)\begin{pmatrix}
			0 & 0\\
			c_{j_0}\e^{-2\i t\theta(z_{j_0})}T(z_{j_0})^{2} & 0
		\end{pmatrix},\nonumber\\
		&\Res_{z=\bar{z}_{j_0}}m^{(1)}(z)=\lim_{z\rightarrow\bar{z}_{j_0}}m^{(1)}(z)\begin{pmatrix}
			0 &  \bar{c}_{j_0}\e^{-2\i t\theta(z_{j_0})}\bar{T}(z_{j_0})^{2} \\
			0& 0
		\end{pmatrix},\nonumber\\
		&\Res_{z=-\bar{z}_{j_0}}m^{(1)}(z)=\lim_{z\rightarrow-\bar{z}_{j_0}}m^{(1)}(z)\begin{pmatrix}
			0 & 0\\
			-\bar{c}_{j_0}\e^{-2\i t\theta(z_{j_0})}\bar{T}(z_{j_0})^{2} & 0
		\end{pmatrix},\nonumber\\
		&\Res_{z=-z_{j_0}}m^{(1)}(z)=\lim_{z\rightarrow-z_{j_0}}m^{(1)}(z)\begin{pmatrix}
			0 &  -c_{j_0}\e^{-2\i t\theta(z_{j_0})}T(z_{j_0})^{2} \\
			0& 0
		\end{pmatrix}.\nonumber
	\end{align}
	\item  $m^{(1)}(z)$ satisfies the  symmetries: $m^{(1)}(z)=\sigma_1\overline{m^{(1)}(\bar{z})}\sigma_1=z^{-1}m^{(1)}(z^{-1})\sigma_2=\overline{m^{(1)}(-\bar{z})}$.
\end{enumerate}
\end{RHP}

\begin{proof}
	Now we prove that $m^{(1)}(z)$ satisfies the above RH problem.
	The first statement and the asymptotics as $z\rightarrow\infty$ in RH problem \ref{rhp1} can be obtained directly from 
	RH problem \ref{rhpm}, we just prove the asymptotics as $z\rightarrow0$. Using the symmetry in (\ref{symmetriesT}) and the expansion in  (\ref{expansionT}), we have
	\begin{align*}
		zm^{(1)}&=T(\infty)^{-\sigma_3}zm(z)T(z)^{\sigma_3}=T(\infty)^{-\sigma_3}(\sigma_2+\mathcal{O}(z))T(z^{-1})^{-\sigma_3}\\
		&=T(\infty)^{-\sigma_3}(\sigma_2+\mathcal{O}(z))(T(\infty)+\mathcal{O}(z))^{-\sigma_3}=\sigma_2+\mathcal{O}(z).
	\end{align*}
	The jump relation is also a direct inference of (\ref{G}) and  RH Problem \ref{rhpm}. Now we derive the residue conditions. For $j_0\in\nabla\cap\Lambda$,
	\begin{align*}
		\Res_{z=z_{j_0}}m^{(1)}&=\Res_{z=z_{j_0}}T(\infty)^{-\sigma_3}m(z)T(z)^{\sigma_3}\\
		&=\lim_{z\rightarrow z_{j_0}}T(\infty)^{-\sigma_3}m(z)T(z)^{\sigma_3}T(z)^{-\sigma_3}\begin{pmatrix}
			0 & 0\\
			c_j\e^{-2\i t\theta(z_{j_0})} & 0
		\end{pmatrix}T(z)^{\sigma_3}\\
		&=\lim_{z\rightarrow z_{j_0}}m^{(1)}(z)\begin{pmatrix}
			0 & 0\\
			c_{j_0}\e^{-2\i t\theta(z_{j_0})}T(z_{j_0})^{2} & 0
		\end{pmatrix}.
	\end{align*}
	For $j_0\in\triangle\cap\Lambda$, we have
	\begin{align*}
		\Res_{z=z_{j_0}}m^{(1)}&=\lim_{z\rightarrow z_{j_0}}(z-z_{j_0})T(\infty)^{-\sigma_3}\left(\frac{m_1^+(z)T(z)}{a(z)},\frac{m_2^-(z)}{T(z)}\right)=T(\infty)^{-\sigma_3}\left(0,\frac{m_2^-(z_{j_0})}{T'(z_{j_0})}\right)\\
		&=\lim_{z\rightarrow z_{j_0}}T(\infty)^{-\sigma_3}\left(\frac{m_1^+(z)T(z)}{a(z)},\frac{m_2^-(z)}{T(z)}\right)\begin{pmatrix}
			0 & c_{j_0}^{-1}\e^{2\i t\theta(z_{j_0})}T'(z_{j_0})^{-2}\\
			0 & 0
		\end{pmatrix},
	\end{align*}
	where we used
	\begin{align*}
		m_1^{(1)}(z_{j_0})&=T(\infty)^{-\sigma_3}\lim_{z\rightarrow z_{j_0}}m_1(z)T(z)=T(\infty)^{-\sigma_3}\lim_{z\rightarrow z_{j_0}}[m_1(z)(z-z_{j_0})]\frac{T(z)-T(z_{j_0})}{z-z_{j_0}}\\
		&=T(\infty)^{-\sigma_3}\Res_{z=z_{j_0}}m_1(z)T'(z_{j_0})=T(\infty)^{-\sigma_3}c_{j_0}e^{-2it\theta(z_{j_0})}m_2(z_{j_0})T'(z_{j_0}).
	\end{align*}
	The others can be obtained using a similar method.
	At last we prove the symmetries for $m^{(1)}(z)$:
	\begin{align*}
		&\overline{m^{(1)}(\bar{z})}=\bar{T}(\infty)^{-\sigma_3}\overline{m}(\bar{z})\overline{T}(\bar{z})^{\sigma_3}=\bar{T}(\infty)^{-\sigma_3}\sigma_1m(z)\sigma_1T(z)^{\sigma_3}=\sigma_1m^{(1)}(z)\sigma_1;\\
		&m^{(1)}(z^{-1})=T(\infty)^{-\sigma_3}m(z^{-1})T(z^{-1})^{\sigma_3}=zT(\infty)^{-\sigma_3}m(z)\sigma_2T(z)^{-\sigma_3}=zm^{(1)}(z)\sigma_2,\\
		&\overline{m^{(1)}(-\bar{z})}=\overline{T(\infty)}^{-\sigma_3}\overline{m(-\bar{z})}\overline{T(-\bar{z})}^{\sigma_3}=T(\infty)^{-\sigma_3}m(z)T(z)^{\sigma_3}=m^{(1)}(z).
	\end{align*}
\end{proof}

\subsection{Opening $\bar{\partial}$ lenses}
In this section, we aim to eliminate the jump on the real axis by choosing a suitable angle, ensuring that the lenses avoid the poles' surrounding disks.

First we are going to show that there's no phase point on the real axis when $-6<\xi<-2$.
\begin{proposition}
	When $|\xi+4|<2 $, there's no phase point in the real axis.
\end{proposition}
\begin{proof}
	From (\ref{theta}), we have
	\begin{equation}
		\theta'(z)=\frac{1}{2}\left\{3(z^2+\frac{1}{z^4})+(\xi+3)(\frac{1}{z^2}+1)\right\}.
	\end{equation}
	Assume that $\theta'(z)$ has zeros in the real axis, then
	\begin{equation*}
		3(z^3+\frac{1}{z^3})+(\xi+3)(\frac{1}{z}+z)=0.
	\end{equation*}
	Let $s=z+1/z\in\mathbb{R}$, then $s\in(-\infty,-2]\cup[2,\infty)$ and the above equation becomes $s^3+(\xi/3-2)s=0$, which means that $s=0$ or $s^2=2-\xi/3$. While $s^2\in(4/3,4)$ as $\xi\in(-6,-2)$, which contradicts the fact that $s^2\geq4$. So there's no phase point on the real axis.
\end{proof}

 \begin{remark}
	The above proof implies that the range for $\xi$ in which there's no phase point on the axis can be extended to ($-6,6$). But as $\xi\in(-2,6)$, $\Lambda$ will always be an empty set, so here we just investigate the case as $\xi\in(-6,-2)$.
\end{remark}

We then define a sufficiently small angle $\theta_0>0$ so that the cone  $\Big\{z\in\mathbb{C}: |\frac{\Re z}{z}|>\cos\theta_0\Big\}$ has no intersection with the disks $|z\pm z_k|\leq\rho$ or $|z\pm\bar{z}_k|\leq\rho$. For any $\xi\in(-6,-2)$, let
\begin{equation}
	\phi(\xi)=\min\left\{\theta_0,\frac{1}{2}\arccos\frac{-4-6\xi-|\xi+4|}{12}\right\},
\end{equation}
and $\Omega=\bigcup_{k=1}^4\Omega_k$, where
\begin{align*}
	&\Omega_1=\{z:\arg z\in(0,\phi(\xi))\},\hspace{0.5cm}\Omega_2=\{z:\arg z\in(\pi-\phi(\xi),\pi)\},\\
	&\Omega_3=\{z:\arg z\in(-\pi,-\pi+\phi(\xi))\},\hspace{0.5cm}\Omega_4=\{z:\arg z\in(-\phi(\xi),0)\}.
\end{align*}
Let the boundaries of $\Omega$ be
\begin{align*}
	&\Sigma_1=\e^{\i\phi(\xi)}\mathbb{R}^+,\quad\Sigma_2=\e^{\i(\pi-\phi(\xi))}\mathbb{R}^+,\\
	&\Sigma_3=\e^{-\i(\pi-\phi(\xi))}\mathbb{R}^+,\quad\Sigma_4=\e^{-\i\phi(\xi)}\mathbb{R}^+,
\end{align*}
which can be seen in Figure \ref{Sigma3}.

\begin{figure}
	\centering
	{\includegraphics[width=0.5\linewidth]{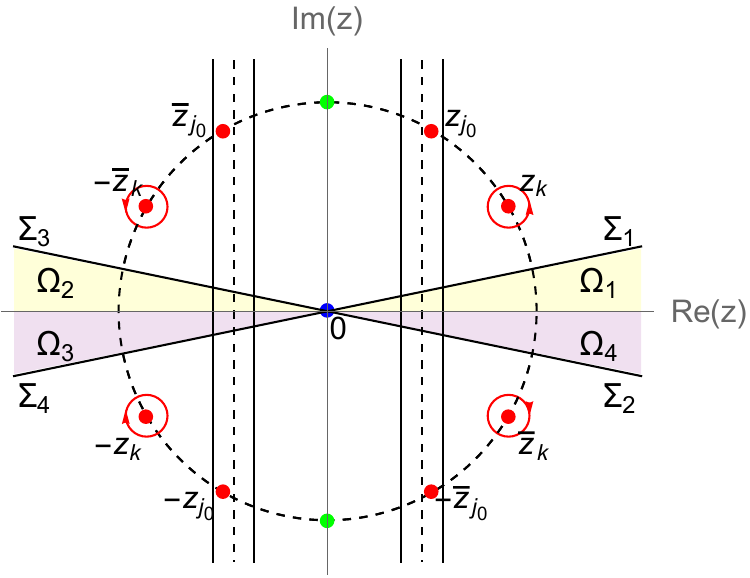}}
	\caption{Find a $\phi(\xi)$ small enough so that there is no pole in the cone and the four rays $\Sigma_k, k=1,2,3,4$ can't intersect with any disks $|z\pm z_k|\leq\rho$ or $|z\pm \bar{z}_k|\leq\rho$.}\label{Sigma3}
\end{figure}

\begin{proposition}
	Let $\xi\doteq\frac{x}{t}$ and $-6<\xi<-2$. Then for $z=|z|\e^{\i\omega}=u+\i v$ and $F(s)=s+s^{-1}$, the phase function $\theta(z;x,t)$ defined in (\ref{theta}) satisfies the following inequalities:
	\begin{align}
		&\Re[2\i t\theta(z;x,t)]\leq -\frac{1}{6}F(|z|)^2t|\sin\omega|(2-|\xi+4|),\quad z\in\Omega_1\cup\Omega_2;\\
		&\Re[2\i t\theta(z;x,t)]\geq \frac{1}{6}F(|z|)^2t|\sin\omega|(2-|\xi+4|),\quad z\in\Omega_3\cup\Omega_4.
	\end{align}
\end{proposition}
\begin{proof}
	We just prove when $z\in\Omega_1$. We can calculate from ($\ref{theta}$), for $z=|z|\e^{\i\omega}$,
	\begin{align*}
		\Re(2\i\theta)=-F(|z|)\sin\omega[\xi+(F(|z|)^2-2)(1+2\cos2\omega)].
	\end{align*}
	Note that $F(|z|)\geq2$, so we have
	\begin{align*}
		\Re(2i\theta)&\leq-F(|z|)|\sin\omega|\left[\xi+(F(|z|)^2-2)\frac{F(|z|)(2-|\xi+4|)-6\xi}{6(F(|z|)^2-2)}\right]\\
		&=-\frac{1}{6}F(|z|)^2|\sin\omega|(2-|\xi+4|).
	\end{align*}
\end{proof}
\begin{proposition}
 Define the functions $R_j:\overline{\Omega}_j\cup\overline{\Omega}_{j+1}\rightarrow\mathbb{C}$, $j=1,2$, satisfying the following boundary conditions,
	\begin{align*}
		\left\{
		\begin{aligned}
			&R_1(z)=\frac{\bar{r}(z)}{1-|r(z)|^2}T_+^{-2}(z),& &z\in\mathbb{R};\\
			&R_1(z)=0, & & z\in\Sigma_1\cup\Sigma_2;
		\end{aligned}\right.\hspace{0.5cm}\left\{
		\begin{aligned}
			&R_2(z)=\frac{r(z)}{1-|r(z)|^2}T_-^{2}(z),& &z\in\mathbb{R};\\
			&R_2(z)=0, & & z\in\Sigma_3\cup\Sigma_4.
		\end{aligned}\right.
	\end{align*}
	For a fixed constant $c_1(q_0)$ and a fixed cutoff function $\varphi\in C_0^{\infty}(\mathbb{R}, [0,1])$, we have the following estimates
	\begin{align}
		&|\bar{\partial}R_j(z)|\leq c_1|z|^{-1/2}+c_1|r'(|z|)|+c_1\varphi(|z|),\hspace{0.5cm}z\in\Omega_i, \hspace{0.5cm}i=1,2,3,4, j=1,2;\label{dbarR1}\\
		&|\bar{\partial}R_j(z)|\leq c_1|z-1|,\hspace{0.5cm}z\in\Omega_1,\Omega_4;  \qquad  |\bar{\partial}R_j(z)|\leq c_1|z+1|,\hspace{0.5cm}z\in\Omega_2,\Omega_3.\label{dbarR2}
	\end{align}
Extending $R$ by $R(z)|_{z\in\Omega_j\cup\Omega_{j+1}}=R_j(z)$, such that the symmetry $R(z)=-\overline{R(\bar{z}^{-1})}$ holds.
\end{proposition}
\begin{proof}
	We just prove the case for $R_1$ in $\overline{\Omega}_1$. The proof for other cases is just an analogue.
	
	\noindent From (\ref{scatteringdata1}) and (\ref{r1}), $z=\pm1$ are singularites of the scattering coefficients $a(z)$ and $b(z)$, which  implies that $z=1$ is a singularity of $R_1(z)$. But one can eliminate this singularity by using $T(z)^{-2}$. From (\ref{r<1}) and (\ref{jumpofT}) we have
	\begin{equation}\label{R1}
		\frac{\bar{r}(z)}{1-|r(z)|^2}T_+(z)^{-2}=\overline{\frac{b(z)}{a(z)}}\left(\frac{a(z)}{T_+(z)}\right)^2\doteq\overline{\frac{J_b(z)}{J_a(z)}}\left(\frac{a(z)}{T_+(z)}\right)^2,
	\end{equation}
	where
	\begin{equation}
		J_b(z)=\det[\psi_1^-(z;x,t),\psi_1^+(z;x,t)]\hspace{0.5cm}J_a(z)=\det[\psi_1^+(z;x,t),\psi_2^-(z;x,t)].
	\end{equation}
    Since in the scattering problem, $X$ is traceless, we can then derive that the determinants of the Jost functions $\psi_j^{\pm}(z;x,t)$, $j=1,2$, are independent of $x$.
    The analyticity of the denominator in the r.h.s. of (\ref{R1}) can be obtained owing to Proposition \ref{analydiff} and \ref{propertiesofT}.
	
	We then introduce the cutoff functions $\chi_0$, $\chi_1$ $\in C_0^{\infty}(\mathbb{R},[0,1])$ with small support near $0$ and $1$ respectively. For any sufficiently small $s\in\mathbb{R}$, $1=\chi_0(s)=\chi_1(s+1)$. Moreover, defining $\chi_1(s)=\chi(s^{-1})$ ensures the symmetry which will be useful in the following content. Then we can rewrite $R_1(z)$ in $\mathbb{R}$ as $R_1(z)=R_{11}(z)+R_{12}(z)$ satisfying
	\begin{equation}\label{decompositionofR1}
		R_{11}(z)=(1-\chi_1(z))\frac{\bar{r}(z)}{1-|r(z)|^2}T_+(z)^{-2},\hspace{0.5cm}R_{12}(z)=\chi_1(z)\overline{\frac{J_b(z)}{J_a(z)}}\left(\frac{a(z)}{T_+(z)}\right)^2.
	\end{equation}
	For a fixed small $\delta_0>0$, extending the function $R_{11}(z)$ and $R_{12}(z)$ by
	\begin{align}
		&R_{11}(z)=(1-\chi_1(|z|))\frac{\bar{r}(|z|)}{1-|r(|z|)|^2}T_+(z)^{-2}\cos(k\arg z),\label{R11}\\
		&R_{12}(z)=f(|z|)g(z)\cos(k\arg z)+\frac{\i|z|}{k}\chi_0(\frac{\arg z}{\delta_0})f'(|z|)g(z)\sin(k\arg z),\label{R12}
	\end{align}
	where $f'(s)$ denotes the derivative of $f(s)$ and
	\begin{equation}
		k\doteq\frac{\pi}{2\theta_0},\hspace{0.5cm}g(z)\doteq\left(\frac{a(z)}{T(z)}\right)^2,\hspace{0.5cm}f(s)\doteq\chi_1(s)\overline{\frac{J_b(s)}{J_a(s)}}.\nonumber
	\end{equation}
	Direct calculation shows that $R_1$ defined in this way satisfies the symmetry  $R_1(s)=-\overline{R_1(\bar{s}^{-1})}$.
	
	Now we give the estimates of the $\bar{\partial}$-derivatives of \eqref{R11}-\eqref{R12}. For  $R_{11}$, we have
	\begin{equation}
		\bar{\partial}R_{11}(z)=-\frac{\bar{\partial}\chi_1(|z|)}{T(z)^2}\frac{\overline{r(|z|)}\cos(k\arg z)}{1-|r(|z|)|^2}+\frac{1-\chi_1(|z|)}{T(z)^2}\bar{\partial}\left(\frac{\overline{r(|z|)}\cos(k\arg z)}{1-|r(|z|)|^2}\right).
	\end{equation}
	Note that for the fixed constants $C$ and $c$, $1-|r(z)|^2>c>0$ as $z\in$ supp$(1-\chi_1(|z|))$ and $|T(z)^{-2}|\leq C$ as $z\in\Omega_1\cap$ supp$(1-\chi_1(|z|))$. For $z=u+\i v=\rho \e^{\i\alpha}$, we have $\bar{\partial}=\frac{1}{2}(\partial_u+\i\partial_v)=\frac{\e^{\i\alpha}}{2}(\partial_{\rho}+\frac{\i}{\rho}\partial_{\alpha})$. As $T(z)$ and $g(z)$ are analytic in $\Omega_1$, it follows that
	\begin{equation}
		\left|\frac{\bar{\partial}\chi_1(|z|)}{T(z)^2}\frac{\overline{r(|z|)}\cos(k\arg z)}{1-|r(|z|)|^2}\right|=\left|\frac{\frac{1}{2}\e^{\i\alpha}\chi_1'\bar{r}\cos(k\alpha)}{T(z)^2(1-|r(|z|)|^2)}\right|\leq c_1\varphi(|z|),
	\end{equation}
	for some $\varphi\in C_0^{\infty}(\mathbb{R},[0,1])$ with a small support near $1$ and with $\varphi=1$ on supp$\chi_1$. Using $r(0)=0$ and $r(z)\in H^1(\mathbb{R})$, we have $|r(|z|)|\leq|z|^{1/2}\|r'\|_{L^2(\mathbb{R})}$, then for some fixed constants $C_2$ and $C_3$, we have
	\begin{align*}
		\left|\frac{1-\chi_1(|z|)}{T(z)^2}\bar{\partial}\left(\frac{\overline{r(|z|)}\cos(k\arg z)}{1-|r(|z|)|^2}\right)\right|&=\left|\frac{1-\chi_1(|z|)}{T(z)^2(1-|r(|z|)|^2)}\right|\left|\frac{1}{2}\e^{\i\alpha}(\bar{r}'\cos(k\alpha)-\i k\bar{r}|z|^{-1}\sin(k\alpha))(1-|r(|z|)|^2)\right.\\
		&\left.+\frac{1}{2}\e^{\i\alpha}(r'\bar{r}+\bar{r}'r)\bar{r}\cos(k\alpha)\right|\\
		&\leq C_2|r'(z)|+C_3\frac{|r(z)|}{|z|}\leq C_2|r'(z)|+C_3|z|^{-1/2}.
	\end{align*}
	So we get the estimation for $\bar{\partial}R_{11}(z)$
	\begin{equation*}
		|\bar{\partial}R_{11}(z)|\leq c_1\varphi(|z|)+c_2|r'(z)|+c_3|z|^{-1/2}.
	\end{equation*}
	Now we estimate $|\bar{\partial}R_{12}(z)|$. We have
	\begin{align*}
		\bar{\partial}R_{12}(z)&=\frac{1}{2}\e^{\i\alpha}g(z)\left[f'cos(k\alpha)(1-\chi_0(\frac{\alpha}{\delta_0}))-\frac{\i kf(\rho)}{\rho}\sin(k\alpha)\right.\\
		&\left.+\frac{\i}{k}(\rho f'(\rho))'\sin(k\alpha)\chi_0(\frac{\alpha}{\delta_0})+\frac{\i}{k\delta_0}f'(\rho)\sin(k\alpha)\chi_0'(\frac{\alpha}{\delta_0})\right],
	\end{align*}
	in which $g(z)$ is bounded. So we can state that $|\bar{\partial}R_{12}(z)|\leq c_4\varphi(|z|)$ for a $\varphi\in C_0^{\infty}[\mathbb{R},[0,1]]$ supported near 1 and for a constant $c_4$, thus yielding (\ref{dbarR1}).
	
	Eventually, as $z\rightarrow1$, we have
	\begin{equation*}
		|\bar{\partial}R_{12}(z)|\leq[\sin(k\alpha)+(1-\chi_0(\alpha/\delta_0))]=\mathcal{O}(\alpha),
	\end{equation*}
	so (\ref{dbarR2}) follows immediately.
\end{proof}

Now we define the modified factorization on $\mathbb{R}$ as $V^{(1)}(z)=\widehat{B}(z)\widehat{B}^{-\dag}(z)$,
where
\begin{equation*}
	\widehat{B}(z)=\begin{pmatrix}
		1 & 0\\
		R_2\e^{-2\i t\theta} & 1
	\end{pmatrix},\qquad \widehat{B}^{\dag}(z)=\begin{pmatrix}
		1 & R_1\e^{2\i t\theta}\\
		0 & 1
	\end{pmatrix}.\label{B}
\end{equation*}
Using (\ref{B}), we define $m^{(2)}(z)$ to open the lenses:
\begin{equation}
	m^{(2)}(z)=\left\{\begin{aligned}
		&m^{(1)}(z)\widehat{B}^{\dag}(z), & &z\in\Omega_1\cup\Omega_2;\\
		&m^{(1)}(z)\widehat{B}(z), & &z\in\Omega_3\cup\Omega_4;\\
		&m^{(1)}(z), & &z\in\mathbb{C}\setminus\bar{\Omega}.
	\end{aligned}\right.
\end{equation}
Let
\begin{equation}\label{contour: Sigma2}
	\Sigma^{(2)}=\bigcup_{j\in H\setminus\Lambda}\left\{z\in\mathbb{C}:|z\pm z_j|=\rho\text{ or }|z\pm\bar{z}_j|=\rho\right\}.
\end{equation}
Then $m^{(2)}(z)$ satisfies the following $\bar{\partial}$-RH problem.

\begin{RHP}\label{m2}
	Find a $2\times2$ matrix-valued function $m^{(2)}(z)=m^{(2)}(z;x,t)$ such that
	\begin{enumerate}
		\item $m^{(2)}(z)$ is continuous in $\mathbb{C}\setminus(\Sigma^{(2)}\cup \{ z_{j_0}\})$, with continuous boundary values $m^{(2)}_+(z)$ and $m^{(2)}_-(z)$ on $\Sigma^{(2)}$ from the left and right, respectively.
		\item $m^{(2)}(z)$ has the following asymptotics:
		\begin{align}
			m^{(2)}(z)=I+\mathcal{O}(z^{-1}), \quad z\rightarrow\infty;\qquad zm^{(2)}(z)=\sigma_2+\mathcal{O}(z), \quad z\rightarrow0.
		\end{align}
		\item $m^{(2)}(z)$ satisfies the following jump relation
		$$m^{(2)}_+(z)=m^{(2)}_-(z)V^{(2)}(z),$$
		where
		 \begin{itemize}
			\item for $j\in\triangle\setminus\Lambda$,
			\begin{equation}
				V^{(2)}(z)=\left\{\begin{aligned}
					&\begin{pmatrix}
						1 & -\frac{(z-z_j)\e^{2\i t\theta(z_j)}}{c_j}T^{-2}(z)\\
						0 & 1
					\end{pmatrix},&&|z-z_j|=\rho, \\
					&\begin{pmatrix}
						1 & \frac{(z+\bar{z}_j)\e^{2\i t\theta(z_j)}}{\bar{c}_j}T^{-2}(z)\\
						0 & 1
					\end{pmatrix},& &|z+\bar{z}_j|=\rho,\\
					&\begin{pmatrix}
						1 & 0\\
						-\frac{(z-\bar{z}_j)\e^{-2\i t\theta(\bar{z}_j)}}{\bar{c}_j}T^2(z) & 1
					\end{pmatrix}, & &|z-\bar{z}_j|=\rho, \\
					&\begin{pmatrix}
						1 & 0\\
						\frac{(z+z_j)\e^{-2\i t\theta(\bar{z}_j)}}{c_j}T^2(z) & 1
					\end{pmatrix}, & &|z+z_j|=\rho,
				\end{aligned}\right.
			\end{equation}\label{V21}
			\item for $j\in\nabla\setminus\Lambda$,
			\begin{equation}
				V^{(2)}(z)=\left\{\begin{aligned}
					&\begin{pmatrix}
						1 & 0\\
						-\frac{c_j\e^{-2\i t\theta(z_j)}}{z-z_j}T^2(z) & 1
					\end{pmatrix}, & &|z-z_j|=\rho, \\
					&\begin{pmatrix}
						1 & 0\\
						\frac{\bar{c}_j\e^{-2\i t\theta(z_j)}}{z+\bar{z}_j}T^2(z) & 1
					\end{pmatrix}, & &|z+\bar{z}_j|=\rho, \\
					&\begin{pmatrix}
						1 & -\frac{\bar{c}_j\e^{2\i t\theta(\bar{z}_j)}}{z-\bar{z}_j}T^{-2}(z)\\
						0 & 1
					\end{pmatrix}, & &|z-\bar{z}_j|=\rho,\\
					&\begin{pmatrix}
						1 & \frac{c_j\e^{2\i t\theta(\bar{z}_j)}}{z+z_j}T^{-2}(z)\\
						0 & 1
					\end{pmatrix}, & &|z+z_j|=\rho.
				\end{aligned}\right.
			\end{equation}\label{V22}
		\end{itemize}
		
		\item For $z\in \mathbb{C}\setminus(\Sigma^{(2)}\cup \{ z_{j_0}\})$, we have:
		\begin{equation}
			\bar{\partial}m^{(2)}(z)=m^{(2)}(z)W(z),
		\end{equation}
	where
	\begin{equation}\label{W}
		W(z)=\left\{\begin{aligned}
			&\bar{\partial}\widehat{B}^{\dag}(z)=\begin{pmatrix}
				0 & \bar{\partial}R_1\e^{2\i t\theta}\\
				0 & 0
			\end{pmatrix}, & &z\in\Omega_1\cup\Omega_2\\
			&\bar{\partial}\widehat{B}(z)=\begin{pmatrix}
				0 & 0\\
				\bar{\partial}R_2e^{-2it\theta} & 0
			\end{pmatrix}, & &z\in\Omega_3\cup\Omega_4\\
			&0 & &\text{elsewhere}.
		\end{aligned}\right.
	\end{equation}
	\item  If $\Lambda=\emptyset$, then $m^{(2)}(z)$ is analytic in $\mathbb{C}\setminus(\bar{\Omega}\cup\Sigma^{(2)})$. If there exists $j_0\in\{0,1,...,N-1\}$ such that $|\Re z_{j_0}-\xi_0|\leq\rho$,  then $m^{(2)}(z)$ is meromorphic in $\mathbb{C}\setminus(\overline{\Omega}\cup\Sigma^{(2)})$ with  four simple poles $\pm z_{j_0}$ and $\pm\bar{z}_{j_0}$, satisfying the following residue conditions:
	    \begin{enumerate}
		\item If $j_0\in\triangle$, denoting $C_{j_0}=c_{j_0}^{-1}T'(z_{j_0})^{-2}$, we have
		\begin{align}\label{resm21}
			\begin{split}
				&\Res_{z=z_{j_0}}m^{(2)}(z)=\lim_{z\rightarrow z_{j_0}}m^{(2)}(z\begin{pmatrix}
					0 & C_{j_0}\e^{2\i t\theta(z_{j_0})}\\
					0 & 0
				\end{pmatrix},\quad
			\Res_{z=-z_{j_0}}m^{(2)}(z)=\lim_{z\rightarrow-z_{j_0}}m^{(2)}(z)\begin{pmatrix}
				0 & 0\\
				-C_{j_0}\e^{2\i t\theta(z_{j_0})} & 0
			\end{pmatrix},\\
				&\Res_{z=-\bar{z}_{j_0}}m^{(2)}(z)=\lim_{z\rightarrow-\bar{z}_{j_0}}m^{(2)}(z)\begin{pmatrix}
					0 & -\bar{C}_{j_0}\e^{2\i t\theta(z_{j_0})}\\
					0 & 0
				\end{pmatrix},\quad
			\Res_{z=\bar{z}_{j_0}}m^{(2)}(z)=\lim_{z\rightarrow\bar{z}_{j_0}}m^{(2)}(z)\begin{pmatrix}
				0 & 0\\
				\bar{C}_{j_0}\e^{2\i t\theta(z_{j_0})} & 0
			\end{pmatrix}.
			\end{split}
		\end{align}
		\item If $j_0\in\nabla$, denoting $C_{j_0}=c_{j_0}T(z_{j_0})^2$, we have
		\begin{align}\label{resm22}
			\begin{split}
				&\Res_{z=z_{j_0}}m^{(2)}(z)=\lim_{z\rightarrow z_{j_0}}m^{(2)}(z)\begin{pmatrix}
					0 & 0\\
					C_{j_0}\e^{-2\i t\theta(z_{j_0})} & 0
				\end{pmatrix},\quad \Res_{z=-z_{j_0}}m^{(2)}(z)=\lim_{z\rightarrow-z_{j_0}}m^{(2)}(z))\begin{pmatrix}
				0 & -C_{j_0}\e^{-2\i t\theta(z_{j_0})}\\
				0 & 0
			\end{pmatrix}.\\
				&\Res_{z=\bar{z}_{j_0}}m^{(2)}(z)=\lim_{z\rightarrow\bar{z}_{j_0}}m^{(2)}(z))\begin{pmatrix}
					0 & \bar{C}_{j_0}\e^{-2\i t\theta(z_{j_0})}\\
					0 & 0
				\end{pmatrix},\quad
				\Res_{z=-\bar{z}_{j_0}}m^{(2)}(z)=\lim_{z\rightarrow-\bar{z}_{j_0}}m^{(2)}(z))\begin{pmatrix}
					0 & 0\\
					-\bar{C}_{j_0}\e^{-2\i t\theta(z_{j_0})} & 0
				\end{pmatrix}.
			\end{split}
		\end{align}
	\end{enumerate}
\end{enumerate}
\end{RHP}
\section{The large-time asymptotic analysis}\label{sec:LTA}

\subsection{Asymptotics of $N$-soliton solution}

 We will neglect the $\bar{\partial}$ component of the solution, then the remaining is a new
RH problem with zero $\bar{\partial}$-derivatives in $\Omega$. After that, the small norm theory can be used to analyze the rest problem.

\begin{proposition}
	Let $m^{(sol)}(z)$ represent the new RH problem derived by excluding 
	the $\bar{\partial}$ component of RH problem \ref{m2}. Specifically, $m^{(sol)}(z)$ is the solution to
	$\bar{\partial}$-RH problem \ref{m2} with $W\equiv0$.
	For scattering data $\left\{r(z),\{z_j,c_j\}_{j=0}^{N-1}\right\}$ in RH problem \ref{m2}, 
	$m^{(sol)}(z)$  is equivalent to RH problem \ref{rhpm} with the modified reflectionless scattering data $\{0,\{z_j,\tilde{c}_j\}_{j=0}^{N-1}\}$. Here, the modified connection coefficients $\tilde{c}_j$ are determined by
	\begin{equation}\label{cj1}
		\tilde{c}_j=c_j(x,t)\exp\left(-\frac{1}{\i\pi}\int_{\mathbb{R}}\log(1-|r(s)|^2)(\frac{1}{s-z_j}-\frac{1}{2s})ds\right).
	\end{equation}
\end{proposition}
\begin{proof}
	When $W\equiv0$, the $\bar{\partial}$-RH problem for $m^{(sol)}(z)$ becomes a new RH problem with jump contour $\Sigma^{(2)}$. The next transformation aims to map each circle in $\Sigma^{(2)}$ back to the corresponding poles. This is done to ensure that $\tilde{m}(z)$ possesses simple poles at each $\pm z_k $ or $\pm\bar{z}_k$ in $\mathcal{Z}$. Additionally, the transformation reverses the triangularity induced by (\ref{T}) and (\ref{m1}):
	\begin{equation}
		\tilde{m}(z)=\left[\prod_{k\in\triangle}\left(-|z_k|^2\right)\right]^{\sigma_3}m^{(sol)}(z)F(z)\left[\prod_{k\in\triangle}\frac{(z-z_k)(z+\bar{z}_k)}{(zz_k-1)(z\bar{z}_k+1)}\right]^{-\sigma_3},
	\end{equation}
	where
	\begin{itemize}
		\item for $j\in\triangle\setminus\Lambda$,
		\begin{equation}
			F(z)=\left\{\begin{aligned}
				&\begin{pmatrix}
					1 & \frac{(z-z_j)\e^{2\i t\theta(z_j)}}{c_j}T^{-2}(z)\\
					0 & 1
				\end{pmatrix},&&|z-z_j|<\rho, \\
				&\begin{pmatrix}
					1 & -\frac{(z+\bar{z}_j)\e^{2\i t\theta(z_j)}}{\bar{c}_j}T^{-2}(z)\\
					0 & 1
				\end{pmatrix},& &|z+\bar{z}_j|<\rho,\\
				&\begin{pmatrix}
					1 & 0\\
					\frac{(z-\bar{z}_j)\e^{-2\i t\theta(\bar{z}_j)}}{\bar{c}_j}T^2(z) & 1
				\end{pmatrix}, & &|z-\bar{z}_j|<\rho, \\
				&\begin{pmatrix}
					1 & 0\\
					-\frac{(z+z_j)\e^{-2\i t\theta(\bar{z}_j)}}{c_j}T^2(z) & 1
				\end{pmatrix}, & &|z+z_j|<\rho,
			\end{aligned}\right.
		\end{equation}
		\item for $j\in\nabla\setminus\Lambda$,
		\begin{equation}
			 F(z)=\left\{\begin{aligned}
				&\begin{pmatrix}
					1 & 0\\
					\frac{c_je^{-2it\theta(z_j)}}{z-z_j}T^2(z) & 1
				\end{pmatrix}, & &|z-z_j|<\rho, \\
				&\begin{pmatrix}
					1 & 0\\
					-\frac{\bar{c}_j\e^{-2\i t\theta(z_j)}}{z+\bar{z}_j}T^2(z) & 1
				\end{pmatrix}, & &|z+\bar{z}_j|<\rho, \\
				&\begin{pmatrix}
					1 & \frac{\bar{c}_j\e^{2\i t\theta(\bar{z}_j)}}{z-\bar{z}_j}T^{-2}(z)\\
					0 & 1
				\end{pmatrix}, & &|z-\bar{z}_j|<\rho,\\
				&\begin{pmatrix}
					1 & -\frac{c_j\e^{2\i t\theta(\bar{z}_j)}}{z+z_j}T^{-2}(z)\\
					0 & 1
				\end{pmatrix}, & &|z+z_j|<\rho.
			\end{aligned}\right.
		\end{equation}
	\item Elsewhere, $F(z)=I$.
\end{itemize}
Note that there are some remarkable properties for $\widetilde{m}$: (i) $\widetilde{m}$ keeps the normalization conditions as $z\to0$ and $z\to\infty$;
(ii) the jump around the simple poles no longer exists;
(iii) $\widetilde{m}$ satisfies the residue conditions (\ref{resm}), with the modified connection coefficients $\widetilde{c}_j$.
Therefore, we verified that $\widetilde{m}$ is the solution of RH problem \ref{rhpm} with the reflectionless scattering data $\{0,\{z_k,\tilde{c}_k\}_{k=0}^{N-1}\}$.
The symmetry $r(s^{-1})=-\overline{r(s)}$, $s\in\mathbb{R}$, gives the modified connection coefficients $\tilde{c}_j=z_j|\tilde{c}_j|$.
In a word, $m^{(sol)}(z)$ is the solution of RH problem \ref{rhpm}, with a $N$-soliton, reflectionless, potential $\tilde{q}(x,t)$, and the discrete spectrum $\mathcal{Z}$, but with the modified connection coefficients $\tilde{c}_j$.
\end{proof}

\begin{proposition}
	Let $\xi=\frac{x}{t}$ and $j_0=j_0(\xi)\in\{-1,0,1,...,N-1\}$. Assume that $m^{\Lambda}(z)$ solves RH problem \ref{m2} with  $W(z)\equiv0$ and $V^{(2)}\equiv I$.
	Then there exists an  unique solution $m^{\Lambda}(z)$ to the above RH problem as follows:
	\begin{itemize}
		\item if $j_0(\xi)=-1$, corresponding to $\Lambda=\varnothing$, then all the discrete eigenvalues $\pm z_j$ are away from the critical lines. Moreover, 
		\begin{equation}\label{mL1}
			m^{\Lambda}(z)=I+\frac{\sigma_2}{z};
		\end{equation}
		\item if $j_0(\xi)\in\nabla$, then
		\begin{align}
			&m^{\Lambda}(z)=I+\frac{\sigma_2}{z}+\begin{pmatrix}
				\frac{\alpha_{j_0}^{\nabla}(x,t)}{z-z_{j_0}}-\frac{\sigma\bar{\alpha}_{j_0}^{\nabla}(x,t)}{z+\bar{z}_{j_0}} & \frac{\bar{\beta}_{j_0}^{\nabla}(x,t)}{z-\bar{z}_{j_0}}-\frac{\sigma\beta_{j_0}^{\nabla}(x,t)}{z+z_{j_0}}\\
				\frac{\beta_{j_0}^{\nabla}(x,t)}{z-z_{j_0}}-\frac{\sigma\bar{\beta}_{j_0}^{\nabla}(x,t)}{z+\bar{z}_{j_0}} & \frac{\bar{\alpha}_{j_0}^{\nabla}(x,t)}{z-\bar{z}_{j_0}}-\frac{\sigma\alpha_{j_0}^{\nabla}(x,t)}{z+z_{j_0}}\\
			\end{pmatrix},\nonumber\\
			&\alpha_{j_0}^{\nabla}(x,t)=-\i z_{j_0}\bar{\beta}_{j_0}^{\nabla},\nonumber\\
			&\beta_{j_0}^{\nabla}(x,t)=\left\{\begin{aligned}
				&\frac{\sin\theta_{j_0}(1+\tanh\varphi_{j_0})(z_{j_0}\sech\varphi_{j_0}-\frac{1}{2}\cos\theta_{j_0}(1+\tanh\varphi_{j_0}))}
				{\left(\frac{1}{2}(1+\tanh\varphi_{j_0})\cos\theta_{j_0}-\sec\theta_{j_0}\sech\varphi_{j_0}\right)^2+\tan^2\theta_{j_0}\sech^2\varphi_{j_0}},& &\sigma=0\\
				&\frac{\i(1+\tanh\varphi_{j_0})}{\sech\varphi_{j_0}-2(1+\tanh\varphi_{j_0})},& &\sigma=1
			\end{aligned}\right.\label{betanabla}
		\end{align}
		where $\theta_{j_0}=\arg z_{j_0}$ and  when $\sigma=1$, $z_{j_0}\neq \i$ and when $\sigma=0$, $z_{j_0}=\i$.
		\item if $j_0(\xi)\in\triangle$, then
		\begin{align}
			&m^{\Lambda}(z)=I+\frac{\sigma_2}{z}+\begin{pmatrix}
				\frac{\alpha_{j_0}^{\triangle}(x,t)}{z-\bar{z}_{j_0}}-\frac{\bar{\alpha}_{j_0}^{\triangle}(x,t)}{z+z_{j_0}} & \frac{\bar{\beta}_{j_0}^{\triangle}(x,t)}{z-z_{j_0}}-\frac{\beta_{j_0}^{\triangle}(x,t)}{z+\bar{z}_{j_0}}\\
				\frac{\beta_{j_0}^{\triangle}(x,t)}{z-\bar{z}_{j_0}}-\frac{\bar{\beta}_{j_0}^{\triangle}(x,t)}{z+z_{j_0}} & \frac{\bar{\alpha}_{j_0}^{\triangle}(x,t)}{z-z_{j_0}}-\frac{\alpha_{j_0}^{\triangle}(x,t)}{z+\bar{z}_{j_0}}\\
			\end{pmatrix},\nonumber\\
			&\alpha_{j_0}^{\triangle}(x,t)=\i\bar{z}_{j_0}\bar{\beta}_{j_0},\nonumber\\
			&\beta_{j_0}^{\triangle}(x,t)=\frac{\sin\theta_{j_0}(1-\tanh\varphi_{j_0})(\bar{z}_{j_0}\sech\varphi_{j_0}-\frac{1}{2}
				\cos\theta_{j_0}(1-\tanh\varphi_{j_0}))}{\left(\frac{1}{2}(1-\tanh\varphi_{j_0})\cos\theta_{j_0}-\sec\theta_{j_0}
				\sech\varphi_{j_0}\right)^2+\tan^2\theta_{j_0}\sech^2\varphi_{j_0}}.\label{betatriangle}
		\end{align}
	\end{itemize}
In case $2$ and $3$, the real phase $\varphi_{j_0}$ is given by
\begin{align}
	&\varphi_{j_0}=2\Im z_{j_0}(x+(4(\Re z_{j_0})^2+2)t+x_{j_0}),\\
	&x_{j_0}=\frac{1}{2\Im z_{j_0}}\left\{\log\left(\frac{|c_{j_0}|}{\Im z_{j_0}}\prod_{k\in\triangle,k\neq j_0}\left|\frac{(z_{j_0}-z_k)(z_{j_0}+\bar{z}_k)}{(z_{j_0}z_{k}-1)(z_{j_0}\bar{z}_{k}+1)}\right|\right)-\frac{\Im z_{j_0}}{\pi}\int_{\mathbb{R}}\frac{\log(1-|r(s)|^2)}{|s-z_{j_0}|^2}ds\right\}.\nonumber
\end{align}
Moreover, as $z\rightarrow\infty$, we have the following asymptotics for $m^{\Lambda}(z)$
\begin{equation}
	m^{\Lambda}(z)=I+\frac{1}{z}\begin{pmatrix} -\i+\bar{\beta}_{j_0}-\sigma\beta_{j_0}\\
		\i+\beta_{j_0}-\sigma\bar{\beta}_{j_0} & \bar{\alpha}_{j_0}-\sigma\alpha_{j_0}
	\end{pmatrix}+\mathcal{O}(z^{-2}),
\end{equation}
from which we can get the soliton solution 
\begin{align}\label{sol}
	\begin{split}
		&sol(z_{j_0},x-x_{j_0},t)=\lim_{z\rightarrow\infty}\i zm_{21}^{\Lambda}(z)=-1+\i(\beta_{j_0}-\sigma\bar{\beta}_{j_0})\\
		&=\left\{\begin{aligned}
			&-1-\frac{2\sin^2\theta_{j_0}\sech\varphi_{j_0}(1+\tanh\varphi_{j_0})}{\left(\frac{1}{2}(1-\tanh\varphi_{j_0})\cos\theta_{j_0}-\sec\theta_{j_0}\sech\varphi_{j_0}\right)^2+\tan^2\theta_{j_0}\sech^2\varphi_{j_0}},&\text{ as }j_0\in\nabla\\
			&-1-\frac{1+\tanh\varphi_{j_0}}{\sech\varphi_{j_0}-2(1+\tanh\varphi_{j_0})},&\text{ as }\sigma=1\\
			&-1+\frac{2\sin^2\theta_{j_0}sech\varphi_{j_0}(1-\tanh\varphi_{j_0})}{\left(\frac{1}{2}(1-\tanh\varphi_{j_0})\cos\theta_{j_0}-\sec\theta_{j_0}\sech\varphi_{j_0}\right)^2+\tan^2\theta_{j_0}\sech^2\varphi_{j_0}}, &\text{ as }j_0\in\triangle.
		\end{aligned}\right.
	\end{split}
\end{align}
\end{proposition}

\begin{proof}
	Owing to $V\equiv I$ and $W\equiv0$, we have  that $m^{\Lambda}(z)$ is meromorphic with simple poles at $z=0$, $\pm z_{j_0}$ and $\pm\bar{z}_{j_0}$ (as $j_0\neq-1$).
	If $\Lambda=\varnothing$, then ($\ref{mL1}$) can be obtained directly from the asymptotic behaviour of RH problem $\ref{m2}$.
	If $\Lambda\neq\varnothing$, then there must be a $j_0\in\Lambda$, note that $C_0\doteq c_{j_0}T(z_{j_0})^2$ meets the condition $C_0=z_{j_0}|C_0|$ as $c_{j_0}=z_{j_0}|c_{j_0}|$.
	If $j_0\in\nabla$, then $m^{\Lambda}(z)$ coincides with the solution of RH problem $\ref{rhpm}$ with reflectionless scattering data, simple poles at $0$, $\pm z_{j_0}$ and $\pm\bar{z}_{j_0}$, and with connection coefficient $C_0$.
	Then the symmetries (\ref{symmetry3}) which is also satisfied by $m^{\Lambda}$, and the residue conditions ($\ref{resm22}$) suggest that $\alpha_{j_0}^{\nabla}=-iz_{j_0}\bar{\beta}_{j_0}^{\nabla}$ and
	\begin{equation*}
		m^{\Lambda}(z)=I+\frac{\sigma_2}{z}+\begin{pmatrix}
			\frac{\alpha_{j_0}}{z-z_{j_0}} & \frac{\bar{\beta}_{j_0}}{z-\bar{z}_{j_0}}\\
			\frac{\beta_{j_0}}{z-z_{j_0}} & \frac{\bar{\alpha}_{j_0}}{z-\bar{z}_{j_0}}
		\end{pmatrix}+\begin{pmatrix}
			\frac{-\bar{\alpha}_{j_0}}{z+\bar{z}_{j_0}} & \frac{-\beta_{j_0}}{z+z_{j_0}}\\
			\frac{-\bar{\beta}_{j_0}}{z+\bar{z}_{j_0}} & \frac{-\alpha_{j_0}}{z+z_{j_0}}
		\end{pmatrix}.
	\end{equation*}
	Under the notation $\e^{\varphi_{j_0}}=\frac{|c_{j_0}|}{\Im z_{j_0}}\e^{2\Im z_{j_0}(x+(4(\Re z_{j_0})^2+2)t)}$, using the residue conditions ($\ref{resm22}$), we can easily obtain ($\ref{betanabla}$).
	If $j_0\in\triangle$, the calculation is similar, but with $\alpha_{j_0}^{\triangle}=\i\bar{z}_{j_0}\bar{\beta}_{j_0}^{\triangle}$ and
	\begin{equation*}
		m^{\Lambda}(z)=I+\frac{\sigma_2}{z}+\begin{pmatrix}
			\frac{\alpha_{j_0}}{z-\bar{z}_{j_0}} & \frac{\bar{\beta}_{j_0}}{z-z_{j_0}}\\
			\frac{\beta_{j_0}}{z-\bar{z}_{j_0}} & \frac{\bar{\alpha}_{j_0}}{z-z_{j_0}}
		\end{pmatrix}+\begin{pmatrix}
			\frac{-\bar{\alpha}_{j_0}}{z+z_{j_0}} & \frac{-\beta_{j_0}}{z+\bar{z}_{j_0}}\\
			\frac{-\bar{\beta}_{j_0}}{z+z_{j_0}} & \frac{-\alpha_{j_0}}{z+\bar{z}_{j_0}}
		\end{pmatrix}.
	\end{equation*}
\end{proof}

\subsection{Small norm RH problem and estimate of errors   }
In this section, we will analyze a small norm RH problem.

\begin{proposition}
	The jump  matrix $V^{(2)}(z)$ has the following estimate
	\begin{equation}
		\|V^{(2)}(z)-I\|_{L^{p}(\Sigma^{(2)})}\leq c\e^{-2t\rho^2}.
	\end{equation}
\end{proposition}
\begin{proof}
	For $|z-z_j|=\rho$ and $j\in\nabla\setminus\Lambda$, we have
	\begin{align}
		\begin{split}
			\|V^{(2)} (z) -I\|_{L^{\infty}(\Sigma^{(2)})}=\left|-\frac{c_j}{z-z_j}T(z)^2\e^{-2\i t\theta(z_j)}\right|
		\leq c\e^{2t\Im z_j(\xi+4(\Re z_j)^2+2)}\leq c\e^{-2t\rho^2}.
		\end{split}
	\end{align}
	The others can be obtained in a similar manner.
\end{proof}
Define
\begin{equation}
	m^{(err)}(z)=m^{(sol)}(z)m^{\Lambda}(z)^{-1},
\end{equation}
then $m^{(err)}(z)$ satisfies the following RH problem:
\begin{RHP}\label{merr}
	Find a $2\times2$ matrix-valued function $m^{(err)}(z)$ such that
	\begin{enumerate}
		\item $m^{(err)}(z)$ is analytic in $\mathbb{C}\setminus\Sigma^{(2)}$, where $\Sigma^{(2)}$ is defined in \eqref{contour: Sigma2}.
		\item $m^{(err)}(z)$ has the following asymptotics:
		\begin{equation}
			m^{(err)}(z)=I+\mathcal{O}(z^{-1}), \hspace{0.5cm} z\rightarrow\infty.
		\end{equation}
		\item $m^{(err)}(z)$ satisfies the following jump relation:
		\begin{equation}
			m^{(err)}_+(z)=m^{(err)}_-(z)V^{(err)}(z), \hspace{0.5cm} z\in\Sigma^{(2)},
		\end{equation}where the jump matrix is defined as
		\begin{equation}
			V^{(err)}(z)=m^{\Lambda}(z)V^{(2)}(z)m^{\Lambda}(z)^{-1},\label{verr}
		\end{equation}
		and the jump contour is shown in Figure \ref{Sigma33}.
	\end{enumerate}
\end{RHP}
The following proposition gives an estimate for the jump matrix $V^{(err)}$.
\begin{proposition}
	The jump matrix $V^{(err)}(z)$ in (\ref{verr}) satisfies
	\begin{equation}
		\|V^{(err)}(z)-I\|_{L^p(\Sigma^{(2)})}\leq c\e^{-2\rho^2t},\quad1\leq p\leq\infty.
	\end{equation}
	Moreover, the solution of RH problem $\ref{merr}$ exits.
\end{proposition}
\begin{proof}
	For $z\in\Sigma^{(2)}$, we have
	\begin{equation}\label{veer-i}
		|V^{(err)}(z)-I|=|m^{\Lambda}(z)(V^{(err)}(z)-I)m^{\Lambda}(z)^{-1}|\leq c|V^{(2)}-I|\leq c\e^{-2\rho^2t}.
	\end{equation}
	According to Beals-Coifman theory, the jump matrix $V^{(err)}(z)$ has the trivial decomposition
	\begin{equation}
		V^{(err)}(z)=(b_-)^{-1}b_+,\hspace{0.5cm}b_-=I,\hspace{0.5cm}b_+=V^{(err)}(z),\nonumber
	\end{equation}
	so we can define
	\begin{equation}
		(w_e)_-=I-b_-=0,\hspace{0.5cm}(w_e)_+=b_+-I=V^{(err)}-I,\hspace{0.5cm}w_e=(w_e)_++(w_e)_-=V^{(err)}-I,\nonumber
	\end{equation}
	and
	\begin{equation}\label{cwef}
		C_{w_e}f=C_-(f(w_e)_+)+C_+(f(w_e)_-)=C_-(f(V^{(err)}-I)),
	\end{equation}
where $C_-$ is the Cauchy projection operator:
\begin{equation*}
	C_-f(z)=\lim_{z'\rightarrow z\in\Sigma^{(2)}}\frac{1}{2\pi \i}\int_{\Sigma^{(2)}}\frac{f(s)}{s-z'}ds,
\end{equation*}
and $\|C_-\|_{L^2}$ is bounded. Then the solution of RH problem $\ref{merr}$ can be expressed as
\begin{equation}\label{merr2}
	m^{(err)}(z)=I+\frac{1}{2\pi \i}\int_{\Sigma^{(2)}}\frac{\mu_e(s)(V^{(err)}(s)-I)}{s-z}ds,
\end{equation}
where $\mu_e(z)\in L^2(\Sigma^{(2)})$ satisfies $(1-C_{w_e})\mu_e(z)=I$. Using (\ref{veer-i}) and (\ref{cwef}), we can obtain that
\begin{equation}
	\|C_{w_e}\|_{L^2(\Sigma^{(2)})}\leq\|C_-\|_{L^2(\Sigma^{(2)})}\|V^{(err)}(z)-I\|_{L^{\infty}(\Sigma^{(2)})}\leq c\e^{-2\rho^2t}.
\end{equation}
Hence, the resolvent operator $(1-C_{w_e})^{-1}$ exists. Consequently, $\mu_e$ and the solution of 
RH problem $\ref{merr}$ $m^{(err)}(z)$ exist.
\end{proof}

\begin{figure}
	\centering
	{\includegraphics[width=0.4\linewidth]{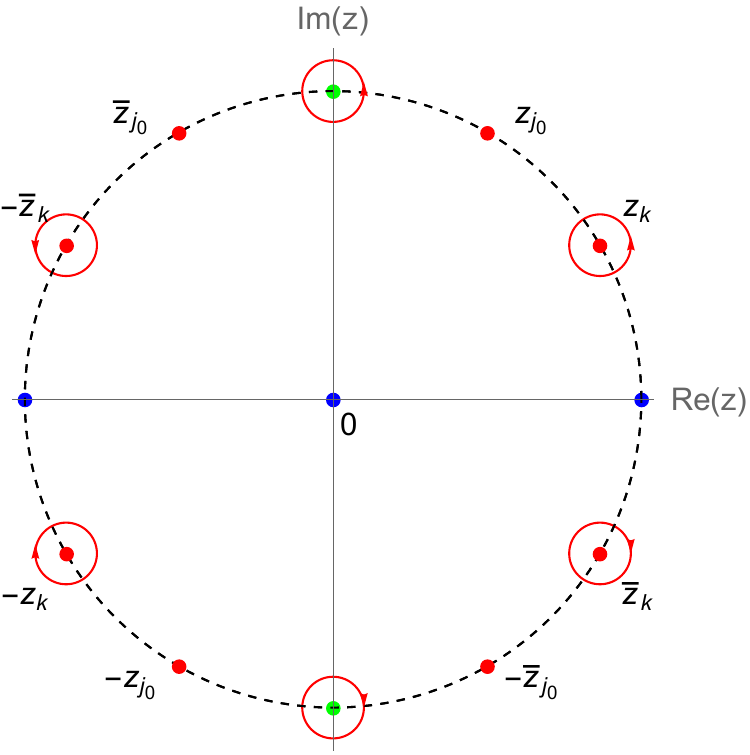}}
	\caption{ The jump contour $\Sigma^{(2)}$  for $m^{(err)}(z)$ is the union of all circle $|z-z_j|=\rho$, $j\in H  \setminus \{j_0\}$.}\label{Sigma33}
\end{figure}

\begin{proposition}
	Let $\xi=\frac{x}{t}$, then for any $(x,t)$ in $\left\{(x,t):-6<\frac{x}{t}<-2\right\}$, as $t\gg1$, uniformly for $z\in\mathbb{C}$ we have the estimate
	\begin{equation}
		m^{(sol)}(z)=m^{\Lambda}(z)[I+\mathcal{O}(\e^{-2\rho^2t})]. \nonumber
	\end{equation}
	Especially, for $z$ sufficiently large, we have the asymptotic extension
	\begin{equation}\label{asympmsol}
		m^{(sol)}(z)=m^{\Lambda}(z)[I+z^{-1}\mathcal{O}(\e^{-2\rho^2t})+\mathcal{O}(z^{-2})].
	\end{equation}
	So we have the asymptotics of the potential
	\begin{equation}\label{qsoln}
		q^{(sol),N}(x,t)=q^{\Lambda}(x,t)+\mathcal{O}(\e^{-2\rho^2t})=sol(z_{j_0},x-x_{j_0},t) +1 + \mathcal{O}(\e^{-2\rho^2t}),
	\end{equation}
	where
	\begin{equation}
		q^{(sol),N}(x,t)=\lim_{z\rightarrow\infty}\i zm^{(sol)}_{21}(z),\quad q^{\Lambda}(x,t)=\lim_{z\rightarrow\infty}\i zm^{\Lambda}_{21}(z).\nonumber
	\end{equation}
\end{proposition}
\begin{proof}
	From ($\ref{merr2}$), we can split $m^{(err)}(z)-I$ into two parts 
	\begin{equation}
		m^{(err)}(z)-I=\frac{1}{2\pi \i}\int_{\Sigma^{(2)}}\frac{V^{(err)}(s)-I}{s-z}ds+\frac{1}{2\pi \i}\int_{\Sigma^{(2)}}\frac{(\mu_e(s)-I)(V^{(err)}(s)-I)}{s-z}ds.\nonumber
	\end{equation}
	It follows that
	\begin{align}
		\begin{split}
			|m^{(err)}(z)-I|&\leq\|V^{(err)}(s)-I\|_{L^2(\Sigma^{(2)})}\left\|\frac{1}{s-z}\right\|_{L^2(\Sigma^{(2)})}
			+\|V^{(err)}(s)-I\|_{L^{\infty}(\Sigma^{(2)})}\|\mu_e(s)-I\|_{L^2(\Sigma^{(2)})}\left\|\frac{1}{s-z}\right\|_{L^2(\Sigma^{(2)})}\\
			&\leq ce^{-2\rho^2t}.\nonumber
		\end{split}
	\end{align}
	Therefore, we have the estimate for $m^{(sol)}(z)$
	\begin{equation*}
		m^{(sol)}(z)=m^{\Lambda}(z)\left[I+\frac{1}{2\pi \i}\int_{\Sigma^{(2)}}\frac{\mu_e(s)(V^{(err)}(s)-I)}{s-z}ds\right]=m^{\Lambda}(z)\left[I+\mathcal{O}(\e^{-2\rho^2t})\right].
	\end{equation*}
	As $z\rightarrow\infty$, $m^{(err)}(z)$ has the following asymptotic extension
	\begin{equation}
		m^{(err)}(z)=I+\frac{m^{(err)}_1}{z}+\mathcal{O}(z^{-1}),
	\end{equation}
where
\begin{align}
	\begin{split}
		m^{(err)}_1=-\frac{1}{2\pi \i}\int_{\Sigma^{(2)}}\mu_e(s)(V^{(err)}(s)-I)ds=-\frac{1}{2\pi \i}\int_{\Sigma^{(2)}}(V^{(err)}(s)-I)ds-\frac{1}{2\pi \i}\int_{\Sigma^{(2)}}(\mu_e(s)-I)(V^{(err)}(s)-I)ds,
	\end{split}
\end{align}
from which we have the following estimate
\begin{equation*}
	m^{(err)}_1(z)\leq\|V^{(err)}(z)-I\|_{L^1}+\|\mu_e(z)-I\|_{L^2}\|V^{(err)}(z)-I\|_{L^2}\leq c\e^{-2\rho^2t}.
\end{equation*}
\end{proof}

\subsection{Analysis on  a pure  $\bar\partial$-problem}

In this section, we use $m^{(sol)}(z)$ to reduce $m^{(2)}(z)$ to a pure $\bar{\partial}$-problem and analyze it.
Define the function
\begin{equation}\label{m3}
	m^{(3)}(z)=m^{(2)}(z)\left(m^{(sol)}(z)\right)^{-1},
\end{equation}
then $m^{(3)}(z)$ satisfies the following $\bar{\partial}$-problem.
\begin{RHP}
	Find a $2\times2$ matrix-valued function $m^{(3)}(z)$ such that
	\begin{enumerate}
		\item $m^{(3)}(z)$ is analytic in $\mathbb{C}\setminus\overline{\Omega}$, and  continuous in $\mathbb{C}$.
		\item $m^{(3)}(z)$ has the asymptotics:
		\begin{equation}
			m^{(3)}(z)=I+\mathcal{O}(z^{-1}),\hspace{0.5cm}z\rightarrow\infty.\label{m3asymp}
		\end{equation}
		\item For $z\in\mathbb{C}$, we have
		\begin{equation}\label{dbarm3}
			\bar{\partial}m^{(3)}(z)=m^{(3)}(z)W^{(3)}(z),
		\end{equation}
		where $W^{(3)}(z)=m^{(sol)}(z)W(z)\left(m^{(sol)}(z)\right)^{-1}$ and $W(z)$ is defined in ($\ref{W}$).
	\end{enumerate}
\end{RHP}

\begin{proof}
	The definition ($\ref{m3}$) implies that $m^{(3)}(z)$ has no jump on the circles $|z\pm z_j|=\rho$ nor $|z\pm\bar{z}_j|=\rho$ since
	\begin{equation*}
		\left(m^{(3)}_-(z)\right)^{-1}m^{(3)}_+(z)=m^{(sol)}_-(z)V^{(2)}(z)\left(m_+^{(sol)}(z)\right)^{-1}=I.
	\end{equation*}
	The asymptotics (\ref{m3asymp}) and $\bar{\partial}$-derivative (\ref{dbarm3}) can be obtained directly from those of $m^{(2)}(z)$ and $m^{(sol)}(z)$.
	We just need to prove that $m^{(3)}$ defined in \eqref{m3} has no isolated singularities. At $z=0$ we have $\left(m^{(sol)}(z)\right)^{-1}=(1-z^{-2})^{-1}\left(m^{(sol)}(z)\right)^T$, then
	\begin{equation}
		\lim_{z\rightarrow0}m^{(3)}(z)=\lim_{z\rightarrow0}\frac{\left(zm^{(2)}(z)\right)z\left(m^{(sol)}(z)\right)^T}{z^2-1}=I,
	\end{equation}
	which implies that $z=0$ is not a singularity of $m^{(3)}(z)$. $\det m^{(sol)}(z)=1-z^{-2}$ indicates that $z=\pm1$ might be the potential singularities.  
	Applying the symmetries ($\ref{symmetrym}$) to the expansion of $m^{(2)}$ and $m^{(sol)}$, we get for some constants $c_1$ and $c_2$,
	\begin{equation}
		m^{(2)}(z)=\begin{pmatrix}
			c_1 & \mp \i c_1\\
			\pm \i\bar{c}_1 & \bar{c}_1
		\end{pmatrix}+\mathcal{O}(z\mp1),\hspace{0.5cm}\left(m^{(sol)}(z)\right)^{-1}=\frac{\pm1}{2(z\mp1)} \begin{pmatrix}
			c_2 & \pm \i\bar{c}_2\\
			\mp \i c_2 & \bar{c}_2
		\end{pmatrix}^T+\mathcal{O}(1).\nonumber
	\end{equation}
It then follows that
\begin{equation}
	\lim_{z\rightarrow\pm1}m^{(3)}(z)=\mathcal{O}(1),
\end{equation}
which implies that $m^{(3)}(z)$ has no singularities at $z=\pm1$. Now we are going to prove that $z_{j_0}$ is not a singularity of $m^{(3)}$. If $m^{(2)}(z)$ has pole at $z_{j_0}$, $j_0\in\nabla\cap\Lambda$, then we have the residue condition
\begin{equation*}
	\Res_{z=z_{j_0}}m^{(2)}(z)=\lim_{z\rightarrow z_{j_0}}m^{(2)}(z)\mathcal{N}_{j_0},
\end{equation*}
where
\begin{equation*}
	\mathcal{N}_{j_0}=\begin{pmatrix}
		0 & 0\\
		C_{j_0}\e^{2\i t\theta(z_{j_0})} & 0
	\end{pmatrix},
\end{equation*}
so $m^{(2)}$ has the Laurent expansion
\begin{equation*}
	m^{(2)}(z)=\frac{\Res_{z=z_{j_0}}m^{(2)}}{z-z_{j_0}}c(z_{j_0})+\mathcal{O}(z-z_{j_0}),
\end{equation*}
where $c(z_{j_0})$ is a constant matrix.
It follows immediately that
\begin{equation*}
	\Res_{z=z_{j_0}}m^{(2)}(z)=c(z_{j_0})\mathcal{N}_{j_0},
\end{equation*}
which gives another form of expansion
\begin{equation}
	m^{(2)}(z)=c(z_{j_0})\left[I+\frac{\mathcal{N}_{j_0}}{z-z_{j_0}}\right]+\mathcal{O}(z-z_{j_0}).
\end{equation}
Since $m^{(sol)}$ and $m^{(2)}$ have the same residue conditions, owing to $\det m^{(2)}(z)=\det m^{(sol)}(z)=1-z^{-2}$, we have
\begin{equation}
	\left(m^{(sol)}(z_{j_0})\right)^{-1}=\frac{z_{j_0}^2}{z_{j_0}^2-1}\left[I-\frac{\mathcal{N}_{j_0}}{z-z_{j_0}}\right]c(z_{j_0})^T+\mathcal{O}(z-z_{j_0}).
\end{equation}
Taking the above into \eqref{m3}, we  have
\begin{equation}
	m^{(3)}(z)=\frac{z_{j_0}^2}{z_{j_0}^2-1}c(z_{j_0})\left[I+\frac{\mathcal{N}_{j_0}}{z-z_{j_0}}\right]\left[I-\frac{\mathcal{N}_{j_0}}{z-z_{j_0}}\right]c(z_{j_0})^T+\mathcal{O}(1),
\end{equation}
from which we can state that $m^{(3)}(z)$ is bounded near the pole and the pole is removable.
We then give the $\bar{\partial}$ derivative of $m^{(3)}$
\begin{align*}
	\bar{\partial}m^{(3)}(z)&=\bar{\partial}m^{(2)}\left(m^{(sol)}(z)\right)^{-1}=m^{(2)}(z)\bar{\partial}R^{(2)}\left(m^{(sol)}(z)\right)^{-1}\\
	&\left[m^{(2)}(z)\left(m^{(sol)}(z)\right)^{-1}\right]\left[m^{(sol)}(z)\bar{\partial}R^{(2)}\left(m^{(sol)}(z)\right)^{-1}\right]=m^{(3)}(z)W^{(3)}.
\end{align*}
where $W^{(3)}(z)=m^{(sol)}(z)W(z)\left(m^{(sol)}(z)\right)^{-1}$.
\end{proof}

The solution of the pure $\bar{\partial}$ problem can be expressed as
\begin{equation}
	m^{(3)}(z)=I-\frac{1}{\pi}\int\int_{\mathbb{C}}\frac{m^{(3)}(s)W^{(3)}(s)}{s-z}dA(s)
\end{equation}
where $dA(s)$ is the Lebesgue measure in $\mathbb{R}$. And it can also be expressed by operator equation
\begin{equation}
	(I-J)m^{(3)}(z)=I \Longleftrightarrow  m^{(3)}(z)=I+Jm^{(3)}(z),\nonumber
\end{equation}
where $J$ is the Cauchy operator
\begin{equation}\label{J}
	Jf(z)=-\frac{1}{\pi}\iint_{\mathbb{C}}\frac{f(s)W^{(3)}(s)}{s-z}dA(s)=\frac{1}{\pi z}*f(z)W^{(3)}(z).
\end{equation}

Then we will show that $J$ is small-norm for large $t$.
\begin{proposition}\label{operatorJ}
	We have $J$: $L^{\infty}(\mathbb{C})\rightarrow L^{\infty}(\mathbb{C})\cap C^0(\mathbb{C})$ and for fixed $\xi_0\in(0,2)$ there exists a constant $C=C(q_0,\xi_0)$, such that for all $t\gg1$ and for all $|\xi+4|\leq\xi_0$,
	\begin{equation}\label{estimationJ}
		\|J\|_{L^{\infty}(\mathbb{C})\rightarrow L^{\infty}(\mathbb{C})}\leq Ct^{-1/2}.
	\end{equation}
\end{proposition}

\begin{proof}
	Here we just consider the case when $f(z)\in L^\infty(\Omega_1)$. From ($\ref{J}$), we have
	\begin{equation}
		|JF(z)|\leq c\|f\|_{L^{\infty}(\mathbb{C})}\iint_{\mathbb{C}}\frac{|W^{(3)}(s)|}{|s-z|}dA(s),\nonumber
	\end{equation}
	where
	\begin{equation*}
		|W^{(3)}(s)|\leq|m^{(sol)}(s)|^2|1-s^{-2}|^{-1}|W(s)|.
	\end{equation*}
	For $z\in\Omega_1$, there exists a  $C_1$ such that
	\begin{equation}
		\parallel m^{(sol)}\parallel\leq C_1(1+|s|^{-1})=C_1|s|^{-1}\langle s\rangle. \nonumber
	\end{equation}
	Then for a fixed constant $c_1$, we have 
	\begin{equation}\label{W3}
		|W^{(3)}(s)|\leq c_1\langle s\rangle|s-1|^{-1}|\bar{\partial}R_1(s)|\e^{\Re(2\i t\theta(s))}.
	\end{equation}
	Since $\frac{\langle s\rangle}{|1+s|}=\mathcal{O}(1)$ in $\Omega_1$, we have the estimate
	\begin{equation}
		|Jf(z)|\leq c_1\iint_{\Omega_1}\frac{\langle s\rangle|\bar{\partial}R_1(s)|\e^{\Re(2\i t\theta(s))}}{|s-z||s-1|}dA(s)\doteq c_1(I_1+I_2+I_3), \nonumber
	\end{equation}
	where
	\begin{align*}
		&I_1=\iint_{\Omega_1}\frac{\langle s\rangle|\bar{\partial}R_1(s)|e^{\Re(2\i t\theta(s))}\chi_{[0,1)}}{|s-z||s-1|}dA(s),\\
		&I_2=\iint_{\Omega_1}\frac{\langle s\rangle|\bar{\partial}R_1(s)|e^{\Re(2\i t\theta(s))}\chi_{[1,2)}}{|s-z||s-1|}dA(s),\\
		&I_3=\iint_{\Omega_1}\frac{\langle s\rangle|\bar{\partial}R_1(s)|e^{\Re(2\i t\theta(s))}\chi_{[2,\infty)}}{|s-z||s-1|}dA(s),
	\end{align*}
	where $\chi_{[0,1)}(|s|)+\chi_{[1,2)}(|s|)+\chi_{[2,\infty)}(|s|)$ is the partition of unity.
	
	We first estimate $I_3$. Since $\langle s\rangle|s-1|^{-1}\leq\kappa$ for $|s|\geq2$, for a fixed $\kappa$, we only need to prove that
	\begin{equation}\label{I3estimation}
		I_3\leq\kappa\iint_{\Omega_1}\frac{(c_2|s|^{-1/2}+c_2|r'(s)|+c_2\varphi(|s|))\e^{\Re(2\i t\theta)}\chi_{[1,\infty)}(|s|)}{|s-z|}dA(s)\leq ct^{1/2},
	\end{equation}
	for some fixed $c_2$ and $c$.
	Let $s=u+\i v$, since $|\xi+4|\leq\xi_0$ for $\xi_0\in(0,2)$, we have
		$\Re(2\i t\theta(s))\leq-c'tv$.
Let $z=z_R+\i z_I$, $1/q+1/p=1$ and $p>2$.  For the integrals in ($\ref{I3estimation}$) including $f(|s|)=|r'(s)|$ or $f(|s|)=\varphi(|s|)$, we can denote
\begin{align}
	I_{31}\doteq\int_0^{\infty}\e^{-c'tv}dv\int_v^{\infty}\frac{f(|s|)\chi_{[1,\infty)}
		(|s|)}{|s-z|}du \leq\int_0^{\infty}\e^{-c'tv}\|f(|s|)\|_{L^2(v,\infty)}\left\|\frac{\chi_{[1,\infty)}(|s|)}{s-z}\right\|_{L^2(v,\infty)}dv.\label{I31}
\end{align}
While
\begin{align}
	\begin{split}
		\left\|\frac{\chi_{[1,\infty)}(|s|)}{s-z}\right\|^2_{L^2(v,\infty)}&\leq\int_v^{\infty}\frac{1}{|s-z|^2}du\leq\int_{-\infty}^{+\infty}\frac{1}{(u-z_R)^2+(v-z_I)^2}du\\
		&\xlongequal{y=\frac{u-z_R}{v-z_I}}\frac{1}{|v-z_I|}\int_{-\infty}^{+\infty}\frac{1}{1+y^2}dy=\frac{\pi}{|v-z_I|},\nonumber
	\end{split}
\end{align}
and
\begin{align}
	\begin{split}
		\|f(|s|)\|^2_{L^2(v,\infty)}&=\int_v^{\infty}|f(\sqrt{u^2+v^2})|^2du\xlongequal{\tau=\sqrt{u^2+v^2}}\int_{\sqrt{2}v}^{\infty}|f(\tau)|^2\frac{\sqrt{u^2+v^2}}{u}d\tau\\
		&\leq\sqrt{2}\int_{\sqrt{2}v}^{\infty}|f(\tau)|^2d\tau\leq\|f(s)\|^2_{L^2(\mathbb{R})}.
	\end{split}
\end{align}
Putting the above two estimates into ($\ref{I31}$), we then have
\begin{equation}\label{I311}
	I_{31}\leq c\|f\|^2_{L^2(\mathbb{R})}\left[\int_0^{z_I}\frac{\e^{-c'tv}}{\sqrt{z_I-v}}dv+\int_{z_I}^{\infty}\frac{\e^{-c'tv}}{\sqrt{v-z_I}}dv\right].
\end{equation}
Using the inequality $\sqrt{z_I}\e^{-c'tz_Iw}\leq ct^{-1/2}w^{-1/2}$, we can estimate the r.h.s in (\ref{I311})
\begin{equation}
	\int_0^{z_I}\frac{\e^{-c'tv}}{\sqrt{z_I-v}}dv\xlongequal{w=\frac{v}{z_I}}\int_0^1\frac{\sqrt{z_I}\e^{-c'tz_Iw}}{\sqrt{1-w}}dw
	\leq ct^{-1/2}\int_0^1\frac{1}{\sqrt{w(1-w)}}dw\leq ct^{-1/2}, \nonumber
\end{equation}
and
\begin{equation}
	\int_{z_I}^{\infty}\frac{\e^{-c'tv}}{\sqrt{v-z_I}}dv\leq\int_0^{\infty}\frac{\e^{-c'tw}} {\sqrt{w}}dw    =    t^{-1/2}\int_0^{\infty}  \frac{\e^{-c't\lambda}}{\sqrt{\lambda}}d\lambda\leq ct^{-1/2}. \nonumber
\end{equation}
Subsequently, we have 
\begin{equation}
	I_{31}\leq ct^{-1/2}\|f\|_{L^2(\mathbb{R})}.
\end{equation}
Then we estimate the terms in ($\ref{I3estimation}$) including $f(|s|)=|s|^{-1}$. Denote
\begin{align}
	I_{32}\doteq\int_0^{\infty}\e^{-c'tv}dv\int_0^{\infty}\frac{\chi_{[1,\infty)}(|s|)|s|^{-1/2}}{|s-z|}du\leq\int_0^{\infty}\e^{-c'tv}\||s|^{-1/2}\|_{L^p(v,\infty)}\||s-z|^{-1}\|_{L^q(v,\infty)}dv.\label{I32}
\end{align}
Then for $p>2$, we have
\begin{equation}
	\||s|^{-1/2}\|_{L^p(v,\infty)}=v^{1/p-1/2}(\int_1^{\infty}\frac{1}{(1+x)^{p/4}}dx)^{1/p}\leq cv^{1/p-1/2},\nonumber
\end{equation}
similarly, we have
\begin{equation}
	\||s-z|^{-1}\|_{L^q(v,\infty)}\leq c|v-z_I|^{1/q-1}, \quad \text{ where }\frac{1}{q}+\frac{1}{p}=1.
\end{equation}
Putting the above two estimates into ($\ref{I32}$) we have
\begin{equation}\label{I321}
	I_{32}\leq c\left[\int_0^{z_I}\e^{-c'tv}v^{1/p-1/2}|v-z_I|^{1/q-1}dv+\int_{z_I}^{\infty}\e^{-c'tv}v^{1/p-1/2}|v-z_I|^{1/q-1}dv\right].
\end{equation}
Then we give the estimate for the r.h.s in (\ref{I321})
\begin{equation}
	\int_0^{z_I}\e^{-c'tv}v^{1/p-1/2}|v-z_I|^{1/q-1}dv\xlongequal{w=v/z_I}\int_0^1\sqrt{z_I}\e^{-c'tz_Iw}w^{1/p-1/2}|1-w|^{1/q-1}dw\leq ct^{-1/2},\nonumber
\end{equation}
and
\begin{align}
	\begin{split}
		&\int_{z_I}^{\infty}\e^{-c'tv}v^{1/p-1/2}|v-z_i|^{1/q-1}dv\xlongequal{v=z_I+w}\int_0^{\infty}\e^{-c't(z_I+w)}(z_I+w)^{1/p-1/2}w^{1/q-1}dw\\
		&\leq_0^{\infty}\e^{-c'tw}w^{-1/2}dw=t^{-1/2}\int_0^{\infty}y^{-1/2}\e^{-y}dy\leq ct^{-1/2}.\nonumber
	\end{split}
\end{align}
Plugging the above two estimations into ($\ref{I321}$), we have
\begin{equation}
	I_{32}\leq ct^{-1/2}.
\end{equation}
So far we have proved ($\ref{I3estimation}$). Now we are going to estimate $I_2$. According to ($\ref{dbarR2}$), for $|s|\leq2$, we have $|\bar{\partial}R_j(s)|\leq c_1|s-1|$ and $\langle s\rangle<\sqrt{5}$, so we have
\begin{equation}
	I_2\leq\sqrt{5}c_1\iint_{\Omega_1}\frac{\e^{-\Re(2\i t\theta)}\chi_{[1,2)}(|s|)}{|s-z|}dA(s).
\end{equation}
It follows immediately from the estimate of $I_3$  that $I_2\leq ct^{-1/2}$.
Finally, we give the estimation of $I_1$. Let $w=1/\bar{z}$ and $\tau=1/\bar{s}$, then we have
\begin{align}
	\begin{split}
		I_1=\iint_{\Omega_1}\frac{|\partial_{\tau}\bar{R}_1|\e^{\Re(2\i t\theta(|\tau|))}\chi_{[1,\infty)}(|\tau|)}{|\tau^{-1}-w^{-1}||\tau^{-1}-1||\tau|^4}\left|\frac{\partial\tau}{\partial\bar{s}}\right|dA(\tau)=|w|\iint_{\Omega_1}\frac{|\bar{\partial}R_1|\e^{\Re(2\i t\theta(|\tau|))}\chi_{[1,\infty)}(|\tau|)}{|\tau-w||\tau-1|}dA(\tau).
	\end{split}
\end{align}
If $|w|\leq3$, it is obvious that the estimate of $I_1$ becomes that of $I_2$. While if $|w|\geq3$, we have
\begin{align}
	I_1\leq3\iint_{|\tau|\geq\frac{|w|}{2}}\frac{|\bar{\partial}R_1|\e^{\Re(2\i t\theta(|\tau|))}\chi_{[1,\infty)}
		(|\tau|)}{|\tau-w|}dA(\tau)+2\iint_{1\leq|\tau|\leq\frac{|w|}{2}}\frac{|\bar{\partial}R_1|\e^{\Re(2\i t\theta(|\tau|))}\chi_{[1,\infty)}(|\tau|)}{|\tau-1|}dA(\tau).\nonumber
\end{align}
It can be estimated by the same method as before, so that $I_1\leq ct^{-1/2}$. Then we have proved ($\ref{estimationJ}$).
\end{proof}

We now show that the equation
$$m^{(3)}=I+Jm^{(3)}$$
holds in the distributional sense.
In fact, for  test function $\phi\in C_0^{\infty}(\mathbb{C},\mathbb{C})$, the differential equation
\begin{equation}\label{f}
	\bar{\partial}\phi(z)=f(z)
\end{equation}
has a  solution
\begin{equation*}
	\phi(z)=\frac{1}{\pi z}*f(z)=\frac{1}{\pi}\iint_{\mathbb{C}}\frac{f(w)}{z-w}dA(w).
\end{equation*}
Using ($\ref{dbarm3}$) and ($\ref{J}$), we have
\begin{align*}
	\begin{split}
		\iint_{\mathbb{C}}\bar{\partial}m^{(3)}(w)\phi(w)dA(w)&=\iint_{\mathbb{C}}m^{(3)}(w)W^{(3)}(w)\left[\frac{1}{\pi}\iint_{\mathbb{C}}\frac{\bar{\partial}\phi(z)}{z-w}dA(z)\right]dA(w)\\
		&=-\iint_{\mathbb{C}}Jm^{(3)}(z)\bar{\partial}\phi(z)dA(z)=\iint_{\mathbb{C}}\bar{\partial}[Jm^{(3)}(z)]\phi(z)dA(z),
	\end{split}
\end{align*}
where we use the fact that the order of integration can be exchanged since $\frac{m^{(3)}(w)W^{(3)}(w)\bar{\partial}\phi(z)}{w-z}\in L^1(\mathbb{C})$. Therefore, in the distributional sense, we have $\bar{\partial}[m^{(3)}-Jm^{(3)}]=0$, which means that $m^{(3)}(z)=I+Jm^{(3)}(z)$.
Now we expand $m^{(3)}(z;x,t)$ as follows
\begin{equation}\label{asympm3}
	m^{(3)}(z;x,t)=I+\frac{m_1^{(3)}(x,t)}{z}+\mathcal{O}(z^{-1}),
\end{equation}
where
\begin{equation}\label{m31}
	m_1^{(3)}(x,t)=\frac{1}{\pi}\iint_{\mathbb{C}}m^{(3)}(s)W^{(3)}(s)dA(s).
\end{equation}
The following proposition gives the estimate for $m_1^{(3)}(x,t)$.
\begin{proposition}
	For $-6<\xi<-2 $,  there exist constants $t_1$ and $c$ such that $m_1^{(3)}(x,t)$ satisfies:
	\begin{equation}
		|m_1^{(3)}(x,t)|\leq ct^{-1} \ \text{ for } \ |x/t+4|<2 \ \text{ and } \ t\geq t_1.
	\end{equation}
\end{proposition}

\begin{proof}
	Lemma $\ref{operatorJ}$ implies that for $t\gg1$ and  $|\xi+4|<\xi_0$, we have $\|m^{(3)}\|_{L^{\infty}}\leq c$. Using ($\ref{W3}$) and ($\ref{m31}$), we have
	\begin{equation}
		|m_1^{(3)}(x,t)|\leq c\iint_{\Omega_1}\frac{\langle s\rangle|\bar{\partial}R_1|\e^{\Re(2\i t\theta)}}{s-1}dA(s)\leq c(I_1+I_2+I_3),
	\end{equation}
	where
	\begin{align*}
		&I_1=\iint_{\Omega_1}\frac{\langle s\rangle|\bar{\partial}R_1|\e^{\Re(2\i t\theta)}\chi_{[0,1)}(|s|)}{s-1}dA(s),\\
		&I_2=\iint_{\Omega_1}\frac{\langle s\rangle|\bar{\partial}R_1|\e^{\Re(2\i t\theta)}\chi_{[1,2)}(|s|)}{s-1}dA(s),\\
		&I_3=\iint_{\Omega_1}\frac{\langle s\rangle|\bar{\partial}R_1|\e^{\Re(2\i t\theta)}\chi_{[2,\infty)}(|s|)}{s-1}dA(s).
	\end{align*}
	For $|s|\geq2$, using $\langle s\rangle|s-1|^{-1}=\mathcal{O}(1)$, and fixing a $p>2$, $q\in(1,2)$, we have
	\begin{align*}
		I_3&\leq c\iint_{\Omega_1}[|r'(|z|)|+\varphi(|s|)+|z|^{-1/2}]\e^{\Re(2\i t\theta)}\chi_{[1,\infty)}(|s|)dA(s)\\
		&\leq c\int_0^{\infty}\|\e^{-c'tv}\|_{L^2(\max\{v,\frac{1}{\sqrt2}\},\infty)}dv+c\int_0^{\infty}\|\e^{-c'tv}\|_{L^2(\max\{v,\frac{1}{\sqrt2}\},\infty)}\||z|^{-1/2}\|_{L^q(v,\infty)}dv\\
		&\leq c\int_0^{\infty}\e^{-c'tv}(t^{-1/2}v^{-1/2}+t^{-1/p}v^{-1/p+1/q-1/2})dv\leq ct^{-1}.
	\end{align*}
For $s\in[0,2]$, we have $\langle s\rangle\leq\sqrt5$. 
Applying an approach similar to the estimation of $I_3$, we derive the inequality $I_2 \leq ct^{-1}$, where the difference lies in substituting $|r'(|z|)| + \varphi(|s|)$ with the function $f = \chi_{[1,2]}(|s|)$.
For $s\in[0,1]$, variable transformations $w=\bar{z}^{-1}$ and $r=\bar{s}^{-1}$ give that
\begin{equation*}
	I_1=\iint_{\Omega_1}\e^{\Re(2\i t\theta(w))}|\bar{\partial}R_1||w-1|^{-1}\chi_{[1,\infty)}(|w|)|w|^{-1}dA(s)\leq ct^{-1}.
\end{equation*}
So we have proved the estimate.
\end{proof}

\section{Proofs of Theorem \ref{mainresult1} and Theorem \ref{mainresult2}}\label{sec:proofs}

\begin{proof}[Proof of Theorem \ref{mainresult1}.]
		For sufficiently large $z\in\mathbb{C}\setminus\overline{\Omega}$, we have
		\begin{equation}\label{*}
			m(z)=T(\infty)^{\sigma_3}m^{(3)}(z)m^{(sol)}(z)T(\infty)^{-\sigma_3}\left[I-\frac{1}{z}T_1^{-\sigma_3}+\mathcal{O}(z^{-2})\right],
		\end{equation}
		where
		\begin{equation*}
			T_1=\sum_{k\in\triangle}4\i\Im z_k-\frac{1}{2\pi \i}\int_{\mathbb{R}}\log(1-|r(s)|^2)ds.
		\end{equation*}
		Below we discuss the asymptotic behaviour of $q(x,t)$ as $t$ large under different conditions.
		\begin{enumerate}
			\item If $\Lambda=\emptyset$, from ($\ref{mL1}$), ($\ref{asympmsol}$) and ($\ref{asympm3}$), we have
			\begin{equation}
				m(z)=I+\frac{1}{z}T(\infty)^{\sigma_3}\left[\sigma_2+m_1^{(3)}(x,t)-T_1^{-\sigma_3}+\mathcal{O}(\e^{-2\rho^2t})\right]T(\infty)^{-\sigma_3}.
			\end{equation}
			Using the reconstruction formula, we have the following asymptotics for $q(x,t)$
			\begin{equation}
				q(x,t)=-1+\mathcal{O}(t^{-1}).\label{68}
			\end{equation}
			\item If $\Lambda\neq\emptyset$, we have
			\begin{equation}
				m^{(sol)}(z)=I+\frac{m_1^{(sol)}(x,t)}{z}+\mathcal{O}(z^{-2}).\nonumber
			\end{equation}
		Plugging the above and ($\ref{asympm3}$) into ($\ref{*}$), we have
		\begin{equation}
			m(z)=I+\frac{1}{z}T(\infty)^{\sigma_3}[m_1^{(3)}(x,t)+m_1^{(sol)}(x,t)-T_1^{\sigma_3}]T(\infty)^{-\sigma_3}. \nonumber
		\end{equation}
		Again, using the reconstruction formula, we then have
		\begin{equation}\label{q}
			q(x,t)=\lim_{z\rightarrow\infty}\i zm_{21}(z)=T(\infty)^{-2}q^{(sol),N}(x,t)+\mathcal{O}(t^{-1}).
		\end{equation}
		Note that $|z_k|=1$ and $\bar{z}_k^{-1}=z_k$. Then (\ref{Tinfty}) gives us $T(\infty)^{-2}=1$, which directly leads to the
		asymptotic stability
		\begin{equation}
			|q(x,t)-q^{(sol),N}(x,t)|\leq ct^{-1}. \label{612}
		\end{equation}
	\end{enumerate}
\quad Next we show that  the $N$-Soliton solutions for the mKdV equation  have the property of soliton resolution.
Consider the order (\ref{order}),
and the sets$ \nabla =\{j:  j\geq j_0\}$,  $ \triangle=\{j: j<j_0 \}$, and  $\Lambda=\emptyset  \ {\rm or}\  \{j_0\}.  $
We could rewrite the asymptotics of $q^{(sol),N}$ (\ref{qsoln}) in terms of $q^\Lambda(x,t)$
\begin{equation*}
	q^{(sol),N}(x,t))= q^{\Lambda}(x,t) + \mathcal{O}(e^{-2\rho^2t}),
\end{equation*}
which combining with (\ref{612}) gives
\begin{equation*}
	q(x,t)-q^{\Lambda}(x,t)=\left(q(x,t)-q^{(sol),N}(x,t)\right)+\left(q^{(sol),N}(x,t)-q^{\Lambda}(x,t)\right)=\mathcal{O}(t^{-1}).
\end{equation*}
Again by  (\ref{qsoln}), we have
\begin{equation}
	q(x,t)= [sol(z_{j},x-x_j,t)+1] +\mathcal{ O }(t^{-1}), \ \  j=0, \cdots, N-1.\label{614}
\end{equation}
By (\ref{68}) and (\ref{614}), we get soliton resolution of the defocusing mKdV equation
\begin{equation}
	q(x,t)=- 1+\sum_{j=0}^{N-1}[sol(z_{j},x-x_j,t)+1] +\mathcal{O}(t^{-1}).
\end{equation}
\end{proof}

\begin{proof}[Proof of Theorem \ref{mainresult2}]
	For $q_0$ close enough to  $q^{(sol),M}(x,0)$, using the Lipschitz continuity ($\ref{21}$)-($\ref{22}$) in Proposition $\ref{analydiff}$, we can immediately get the relation of poles in ($\ref{M1}$).  Then applying Theorem $\ref{mainresult1}$ to $q_0$, we can obtain ($\ref{qse}$). Finally, simple calculations ($\ref{qse}$) give ($\ref{qse2}$).
\end{proof}


\bigskip\noindent\textbf{\sffamily Acknowledgments}
\par\medskip\noindent
We thank Gino Biondini for many insightful conversations on topics related to the present work.
This work is supported by  the National Natural Science
Foundation of China (Grant No. 11671095,  51879045).


\medskip
\let\em=\it

\makeatletter
\def\@biblabel#1{#1.}
\def\doibase{http://dx.doi.org/}
\def\reftitle#1{``#1''}
\def\booktitle#1{\textit{#1}}
\makeatother

\end{document}